\documentclass[12pt,a4paper]{amsproc}
\usepackage[english]{babel}
\usepackage{latexsym}
\usepackage{amsmath,amssymb,amsthm,color,enumerate,enumitem}
\usepackage{amscd}
\usepackage{xcolor}
\usepackage{todonotes}
\usepackage{marginnote}
\usepackage[colorlinks=true]{hyperref}
\usepackage[cp850]{inputenc}
\usepackage[mathscr]{eucal}
\theoremstyle{plain}

\newtheorem{quest}{Question}[section]

\newtheorem{defin}[quest]{Definition}
\newtheorem{theorem}[quest]{Theorem}
\newtheorem{prop}[quest]{Proposition}
\newtheorem{corollary}[quest]{Corollary}
\newtheorem{lemma}[quest]{Lemma}

\newtheorem*{prop*}{Proposition}
\newtheorem*{defin*}{Definition}
\theoremstyle{remark}
\newtheorem{example}[quest]{Example}
\newtheorem{remark}[quest]{Remark}

\newcommand{\adef}{\begin{defin}}
\newcommand{\zdef}{\end{defin}}

\newcommand{\aproof}{\begin{proof}}
\newcommand{\zproof}{\end{proof}}

\begin{document}

%\date{\today}
\thanks{Grant 23/12916-1, S\~ao Paulo Research Foundation (FAPESP)} 
\author{Leandro Candido}
\address{Instituto de Ci\^encia e Tecnologia da Universidade federal de S\~ao Paulo, Av. Cesare Giulio Lattes, 1201, ZIP 12247-014 S\~ao Jos\'e dos Campos/SP, Brasil}\email{leandro.candido@unifesp.br}

\title{Kurzweil--Stieltjes integration on compact lines}
\author{Pedro L. Kaufmann}
\address{Instituto de Ci\^encia e Tecnologia da Universidade federal de S\~ao Paulo, Av. Cesare Giulio Lattes, 1201, ZIP 12247-014 S\~ao Jos\'e dos Campos/SP, Brasil}\email{plkaufmann@unifesp.br}
\email{}
%\keywords{}
%\subjclass{}

\maketitle

\begin{abstract}
   We develop a version of the Kurzweil--Stieltjes integral on compact lines and establish its fundamental properties. For sufficiently regular integrators, we obtain convergence theorems and show that the presented integration process generalizes Lebesgue integration with respect to positive Radon measures. Additionally, we introduce a notion of derivation on compact lines which, when paired with the proposed integral, yields a formulation of the Fundamental Theorem of Calculus in this context. 
\end{abstract}

\tableofcontents

% \noindent
% \begin{flushright}
% \begin{minipage}{0.44\textwidth}
% % \begin{quote}
% % \textit{What was true yesterday is true today.}
% % \end{quote}
% \end{minipage}
% \end{flushright}

\section{Introduction}

Jaroslav Kurzweil, and independently Ralph Henstock (see \cite{kurzweil1957generalized} and \cite{henstock1968riemann}), discovered a simple way to modify the Riemann integral on the real line, obtaining a powerful analytical tool known nowadays as the \emph{Henstock--Kurzweil integral}. On one hand, this integral generalized the Lebesgue and Newton integrals in a user-friendly way, and on the other hand it presented remarkable features, notably several convergence properties, a very general Fundamental Theorem of Calculus, and closedness with respect to improper integration, a result known as Hake's Theorem. See any textbook on the subject for basic definitions and a more complete account of properties of this integral,  for example \cite{bartle2001modern}. The simplicity of its definition, allied with the wide range of properties, makes it  a useful asset in particular when dealing with generalized ordinary differential equations, see for example \cite{schwabik1992generalized} and \cite{monteiro2019kurzweil}. The feature of the Henstock-Kurzweil integral that perhaps contrast the most with the Lebesgue integral is that it is a \emph{nonabsolute} integral. This means that there are integrable functions $f$ such that $|f|$ is not integrable. The price payed is that we cannot define a norm as directly as in the space of Lebesgue integrable functions - where for instance the seminorm $\|f\| = \int |f|$ and variations of it turn out to be fundamental in the development of functional analysis and measure theory, with a wide range of applications.  

Although we can naturally generalize Henstock-Kurzweil integration for Euclidean domains in more dimensions - partitioning our domain is parallelograms with sides parallel to the coordinate axes, still under the control of a gauge, see \cite{lee2011henstock} - many desirable properties lack for the resulting integral. For instance, it was early noticed that two-dimensional Perron integrability (which is equivalent to Henstock--Kurzweil integrability) is not invariant by rotations, see \cite{kartak1955k}. Several attempts have been made to modify the definition of Henstock--Kurzweil for multidimensional domains in order to obtain better properties, for instance imposing regularity conditions to the parallelograms (\cite{jarnik1983mawhin}) or allowing other shapes to the partition elements (\cite{jarnik1985non},  \cite{bianconi2010triangle}). In his book \cite{pfeffer1993riemann}, Pfeffer analyses some conflicting properties when defining nonabsolute integrals on more than one dimension. 

Henstock-Kurzweil integration works better in one dimensional Euclidean space due to its compatible order. To illustrate this, consider the following classical example. Let $\sum a_n = a$ be a conditionally convergent series of real numbers, and consider $I_1,I_2,\dots$ a countable set of nondegenerate subintervals of $[0,1]$ with $I_1<I_2<\dots$. Let $f=\sum a_n\chi_{I_n}/\ell(I_n)$, where $\chi_{I_n}$ is the characteristic function of $I_n$, and $\ell(I_n)$ is its length. Then $f$ is Henstock--Kurzweil integrable on $[0,1]$, and the value of the integral is $a$. Now, since the series is conditionally convergent, it is a known fact that given an arbitrary real number $b$, there exists a permutation $\sigma$ of the natural numbers such that $\sum a_{\sigma(n)}=b$. It follows that  $\sum a_{\sigma(n)}\chi_{I_n}/\ell(I_n)$ is Henstock--Kurzweil integrable, with integral $b$. This indicates the importance of the order in which the pieces where $f$ has integral $a_n$ appear. This phenomenon is absent for the Lebesgue integral, since $f=\sum a_n\chi_{I_n}/\ell(I_n)$  is Lebesgue integrable if and only if $\sum a_n$ is unconditionally convergent.  

This motivates the investigation of nonabsolute integrals on more general spaces with compatible linear orders. A general ambient space would be the \emph{linearly ordered topological spaces}, which are spaces whose topology is induced by a linear order. In this work, we will investigate a Kurzweil type integral on linearly ordered topological spaces that are moreover Dedekind-complete and bounded. These spaces are also known as \emph{compact lines}. The idea of interpreting the Lebesgue integral of continuous functions using Riemann sums on compact lines has been recently explored in Victor Ronchim's PhD thesis \cite{ronchim2021study} with the objective of studying extension of $c_0(I)$-valued operators on $C(K)$, where $K$ is a compact line. Later on, the results were published in \cite{tausk2023} without using these integrals, but Ronchim's original ideas inspired the current work.   Basic examples of compact lines are the compact real intervals, and more generally bounded time scales, but the class is much wider. Since compact lines in general lack of algebraic structure, we cannot base our definition on Riemann sums of the form $\sum f(t_n)(x_{n-1}-x_n)$. Instead, we use sums of the form  $\sum f(t_n)(G(x_{n-1})-G(x_n))$, for appropriate real-valued functions $G$, giving rise to a Kurzweil--Stieltjes type of integral. As we shall see, this integral share many features with the Henstock--Kurzweil integral on the compact real interval, such as linearity, additivity, and key convergence properties. Additionally, we will show that when we pair our integral with a suitable derivation process that reminisce of the nabla derivative on time scales, a far-reaching Fundamental Theorem of Calculus is obtained. 

The remainder of this work is organized as follows. In Section \ref{sec:cpt lines}, we expose prerequisites on compact lines. In Section \ref{sec:KS integration} we introduce a Kurzweil--Stieltjes integration process on compact lines, and present some examples and basic properties of the resulting integral. Moreover, in this section we explore the connection between amenable integrators and Radon measures. In Section \ref{sec:Calculus}, we pair our integral with a suitable differentiation process and provide results related to the Fundamental Theorem of Calculus. In Section \ref{sec:convergence}, we obtain basic convergence theorems such as the Monotone Convergence Theorem and the Dominated Convergence Theorem. We also establish the relationship between our integral and the Lebesgue integral.

\section{Compact lines}\label{sec:cpt lines}

% \begin{itemize}
%     \item Defini\c c\~ao
%     \item Comentar sobre necessidade de completude na Se\c c\~ao 3.1
%     \item Propriedade de Heine Borel
%     \item Exemplos (time scales, linhas longas, double arrow space)
%     \item Comentar sobre a integral do Ronchim que vai aparecer na Se\c c\~ao 3.3
% \end{itemize}

%%%%%% Deixei comentado, caso decidamos fazer uma vers\~ao mais detalhada

%A \emph{linear order} on a nonempty set $L$ is a binary relation $\leq$ satisfying the following properties for arbitrary $x,y,z\in L$:
%\begin{itemize}
    %\item (reflexivity) $x\leq x$,
    %\item (antissymmetry) $x\leq y$ and $y\leq x$ $\Rightarrow$ $x=y$, 
    %\item (transitivity) $x\leq y$ and $y\leq z$ $\Rightarrow$ %$x\leq z$, and
    %\item (comparability) $x\leq y$ or $y\leq x$. 
%\end{itemize}
%The sets $(x,y)=\{z\in L:x<z<y\}$ induce a topology on $L$, called the \emph{order topology}. A \emph{linearly ordered topological space} is a nonempty set equipped with a linear order and its induced topology. The most basic examples are nonempty subsets of the real line, with their inherited order and topology. 

%In what follows, we will expose basic properties and examples which are relevant for our exposition. 

A \emph{compact line} is a nonempty set equipped with a linear order whose order topology is compact. We begin by recalling some basic properties and examples of compact lines relevant to our purposes, and refer the reader to \cite{engelking1989general}, \cite{steen1978counterexamples}, and \cite{cater2006simple} for a more detailed treatment of linearly ordered topological spaces. 

Compact lines are \emph{Dedekind-complete}, meaning that every nonempty subset has both a supremum and an infimum. In particular, a compact line $K$ has a least and a greatest element, which we denote by $0_K$ and $1_K$, respectively. Intervals in $K$ will be denoted following the usual notation for $\mathbb{R}$; for example, $(x, y] = \{ z \in K \mid x < z \leq y \}$. 

Dedekind-completeness also allows for a classification of the points in $K$ according to whether they are limit points or have immediate predecessors or successors, as we now describe. 

\begin{defin}[Classification of points]\label{def:classification}
    Let $K = [0_K, 1_K]$ be a compact line. A point $x \in [0_K, 1_K)$ is said to be \emph{right-dense} if every neighborhood of $x$ contains a point greater than $x$. Otherwise, we say that $x$ is \emph{right-isolated}, and we define  
    \[x^+ = \inf\{y \in K : y > x\} \in K.\]
    In this case, $x^+$ is called the \emph{successor} of $x$. Similarly, a point $x \in (0_K, 1_K]$ is said to be \emph{left-dense} if every neighborhood of $x$ contains a point less than $x$. Otherwise, we say that $x$ is \emph{left-isolated}, and we define  
    \[x^- = \sup\{y \in K : y < x\} \in K.\]
    We call $x^-$ the \emph{predecessor} of $x$. We always consider $0_K$ to be left-isolated and $1_K$ to be right-isolated.
\end{defin}

Two fundamental classes of compact lines that serve as building blocks for more elaborate constructions are \emph{ordinal intervals} and \emph{bounded time scales}. Ordinal intervals are compact lines of the form $[0,\alpha] = \{\kappa \mid \kappa \text{ is an ordinal and } \kappa \leq \alpha\}$, endowed with the order topology. This class allows us to construct compact lines of arbitrary cardinality.

Bounded time scales, on the other hand, are particularly interesting due to their simplicity and relevance in the theory of dynamic equations, see for instance \cite{bohner2001dynamic}. These are defined as the intersection of a nonempty closed subset $ \mathbb{T} \subset \mathbb{R} $ (called a \emph{time scale}) with a real nondegenerate interval $[a,b]$, provided the intersection, denoted $ [a,b]_\mathbb{T} $, is nonempty. Observe that bounded time scales are nothing but nonempty compact subsets of $ \mathbb{R}$.

We also point out that terminology in the literature on time scales often refers to right- and left-isolated points as \emph{right scattered} and \emph{left scattered}, respectively. We choose instead to adopt the standard topological terminology introduced in Definition~\ref{def:classification}, in part because the term ``scattered'' has a different meaning in general topology, and this could lead to ambiguity.

Once we have a given family of compact lines at our disposal, there are several ways to construct new ones. One simple method is to take a nonempty closed subset of a compact line, which remains a compact line under the induced order. Another important construction arises from the \emph{lexicographic product}: given two compact lines $ K $ and $ L $, the set $ K \times L $ endowed with the lexicographic order $ \leq_\mathrm{lex} $ becomes a compact line. This allows, for instance, the construction of $ \alpha $ copies of a compact line $ K $, arranged lexicographically, by considering
$([0,\alpha] \times K, \leq_\mathrm{lex})$.
As an illustrative aside, this kind of construction appears in the classical definition of \emph{long lines}, where one considers the lexicographic product $[0,\alpha] \times [0,1)$, with $\alpha$ a large ordinal, and removes the smallest point. See~\cite{steen1978counterexamples} for details.

Another notable class of examples is given by the so-called \emph{double-arrow spaces}. Given a compact line $ K $ and a subspace $ L \subset K $, the \emph{Aleksandrov double-arrow space} is defined as
\[
DA(K, L) = \left((K \times \{0\}) \cup (L \times \{1\}), \leq_{\mathrm{lex}}\right),
\]
where the space is endowed with the lexicographic order $ \leq_{\mathrm{lex}} $. In the particular case where both $ K $ and $ L $ are the closed unit interval $[0,1]$ with the usual order, the resulting space is commonly referred to as the \emph{split interval}.

 A remarkable result by Ostaszewski~\cite{ostaszewski1974characterization} asserts that every separable compact line -  hence, in particular, every time scale -  is order-isomorphic to a subspace of the split interval. This insight reinforces our belief that the integral framework developed in this paper may offer a useful perspective on nonabsolute integration theory, potentially applicable to both separable and non-separable settings.

\section{Kurzweil--Stieltjes integration}\label{sec:KS integration}

% \begin{itemize}
%     \item Comentar sobre a constru\c c\~ao, e a necessidade de mudar a defini\c c\~ao de gauge
%     \item Comentar sobre o trip\'e da teoria de integra\c c\~ao de Henstock-Kurzweil: lema de cousin, lema de saks-henstock, e teorema de cobertura de vitali
% \end{itemize}

The development of a satisfactory nonabsolute integration theory on the real line \`a la Henstock-Kurzweil classically depends on three main ingredients: Cousin's Lemma (which guarantees a proper definition of the integral), Saks-Henstock's Lemma (which permits to have control on how integrable functions behave on elementary sets, despite the conditionally convergent aspect of the integral), and Vitali's Covering Theorem (a result of geometric nature used to obtain a very powerful Fundamental Theorem of Calculus). When passing to the more general Kurzweil--Stieltjes setting, the identification of bounded variation integrators with Lebesgue--Stieltjes measures also plays a central role.  It turns out that versions of these four results hold for an appropriate integral on compact lines, under reasonable conditions. In this section, we introduce such a Kurzweil--Stieltjes type integration process on compact lines, discuss its basic properties, and provide examples. In particular, we prove versions of the mentioned four fundamental results that allow us to further develop convergence theorems and calculus results in the following sections. 

%versions of Cousin's and Saks-Henstock's Lemmas, and investigate the connection between integrators and measures. A version of Vitali's Covering Theorem and applications to Calculus are postponed to the next section.  \textcolor{magenta}{Revisar esse par\'agrafo no final}

%The successful development of the theory of Henstock-Kurzweil integration on the compact real interval is sustained by three pillars: Cousin's Lemma, which guarantee the existence of $\delta$-fine partitions given an arbitrary gauge $\delta$, Saks-Henstock's Lemma, which allows some control over how integrable functions behave on elementary sets despite the conditional nature of the integral, and Vitali's Covering Theorem, a result of geometric nature which together with Saks-Henstock's Lemma allows for a powerful Fundamental Theorem of  Calculus, and sheds light on the relation to the Lebesgue integral. As we shall see, we can generalize this type of integral to compact lines in a rather natural way, and still have versions of these three results, without loss of generality on the Fundamental Theorem of Calculus under reasonable conditions. In this section, we will introduce Kurzweil--Stieltjes integration on compact lines, discuss basic properties, and give examples. In particular, we prove Cousin's Lemma and conclude with a version of Saks-Henstock's Lemma. Vitali's Covering Theorem and applications to Calculus are postponed to the next Section. 

\subsection{Gauges and partitions} 

% \begin{itemize}
%     \item Defini\c c\~ao de parti\c c\~ao como um par divis\~ao + tags
%     \item Lema de Cousin
%     \item Exemplo indicando que precisa de completude
% \end{itemize}

We start by defining tagged partitions analogously as for the classical Henstock-Kurzweil integral. 

\begin{defin}[Tagged partitions]\label{def:partition}
    A \emph{tagged partition} $P$ of a compact line $K$ consists on a finite set of points in $K$, $\mathcal D=\{x_0,x_1,\dots,x_n\}$, with $0_K=x_0\leq x_1\leq \dots\leq x_n=1_K$, together with a set of points $\mathcal T=\{t_1,\dots,t_n\}$ such that $t_i\in [x_{i-1},x_i]$, for each $i\in \{1,\dots,n\}$. The elements of $\mathcal D$ are called \emph{division points}, while the elements of $\mathcal T$ are called \emph{tag points}. 
\end{defin}
Since all partitions considered will be tagged, we will omit this qualifier throughout. We will always represent a partition with $\mathcal{D} = \{x_0, x_1, \dots, x_n\}$ and $\mathcal{T} = \{t_1, \dots, t_n\}$, in the form  
\[P = \{([x_{i-1}, x_i], t_i) : 1 \leq i \leq n\}.\]
The pairs $([x_{i-1}, x_i], t_i)$, for $i = 1, \dots, n$, are referred to as the \emph{components} of the partition.

 The definition of gauge - which in the classical theory is a function from $[a,b]$ to $(0,\infty)$ -  can be easily adapted to compact lines once we observe that the job the gauge should perform is essentially assign to each point a neighbourhood of it. The definition of $\delta$-fine partition for a gauge $\delta$ is more subtle, in particular to guarantee their existence for an arbitrary gauge. 

\begin{defin}[Gauges, refinements and $\delta$-fine partitions]
    Let $K$ be a compact line. A \emph{gauge} on $K$ is a function $\delta$ from $K$ into the set of open intervals of $K$ such that $x\in\delta(x)$, for each $x\in K$. Given two gauges $\delta$ and $\eta$ on $K$, we say that $\eta$ \emph{refines} $\delta$ if $\eta(x)\subset \delta (x)$ for each $x\in K$.  Given a gauge $\delta$ on $K$, a partition $P =\{([x_{i-1},x_i],t_i):1\leq i \leq n\}$ of $K$ is said to be \emph{$\delta$-fine} when $(x_{i-1},x_i]\subset \delta(t_i)$, for all $i\in\{1,\dots,n\}$. We denote this relationship by $P \ll \delta$.

    When $I$ is a compact subinterval of $K$, it is clear that $\delta_I : I \to \mathbb{R}$, defined by $\delta_I(x) = \delta(x) \cap I$, is a gauge on $I$ (sometimes referred to as the induced gauge). Whenever $P$ is a partition of $I$, we will simply say that $P$ is $\delta$-fine instead of $\delta_I$-fine, as the context makes the meaning clear and prevents any confusion.
\end{defin}

Some remarks are in order. 
\begin{enumerate}[label=\textnormal{(P\alph*)}]
    \item\label{it:Partition1} By an open interval in $K$, we mean an interval that is also an open subset of $K$. For example, if $x\in (0_K,1_K]$, then $[0_K, x)$ is an open interval. It is worth noting that an interval may admit multiple representations in terms of its endpoints. For instance, if $x \in (0_K, 1_K]$ is a left-isolated point in a compact line $K$ (consider, for example, the ordinal space $[0, \omega]$, which contains infinitely many left-isolated points), then the interval $[x, 1_K]$ can equivalently be written as $(x^-, 1_K]$.
    \item\label{it:Partition2} We formally allow the possibility that $x_{n-1} = t_n = x_n$ in a partition. This does not affect $\delta$-fineness, as $(x_{i-1}, x_i] = \emptyset \subset \delta(t_i)$ in such cases. We do not require $x_{i-1} < x_i$ to accommodate trivial $\delta$-fine partitions, especially in one-point compact lines. Moreover, it is sometimes convenient to insert tags and division points to enlarge a partition. For instance, a component $([x_{i-1}, x_i], t_i)$ may be split into two components, $([x_{i-1}, x'_i], t_i)$ and $([x'_{i-1}, x_i], t_i)$, where $x'_i = t_i = x'_{i-1}$, possibly with $x'_{i-1} = x_i$. 
    
    \item\label{it:Partition3} In classical Henstock-Kurzweil integration theory, one requires that $[x_{i-1}, x_i] \subset (t_i - \delta(t_i), t_i + \delta(t_i))$, which, in the notation used above, corresponds to $[x_{i-1}, x_i] \subset \delta(t_i)$. Our choice of $(x_{i-1}, x_i]$ instead of $[x_{i-1}, x_i]$ is intended to ensure the existence of $\delta$-fine partitions for an arbitrary gauge $\delta$. This is the content of the following lemma.
 
\end{enumerate}

%It will be often useful to refer to the interval $\delta(x)$ in terms of its endpoints. 

\begin{lemma}[Cousin]\label{lem:cousin}
    Let $\delta$ be a gauge on a compact line $K$. Then $K$ admits a $\delta$-fine partition. 
\end{lemma}    
 \begin{proof}
We assume that $0_K < 1_K$, as the conclusion is trivial in the opposite case. Consider the set $S = \{x \in (0_K,1_K] : [0_K, x] \text{ admits a } \delta\text{-fine   partition} \}$. First, we observe that $S \neq \emptyset$. Indeed, if $0_K$ is a right-isolated point, then the partition $P = \{([0_K, 0_K^+], 0_K^+)\}$ is a $\delta$-fine   partition for $[0_K, 0_K^+]$. On the other hand, if $0_K$ is right-dense, then there exists $x \in \delta(0_K) \cap (0_K, 1_K]$, and in this case, the partition $P = \{([0_K, x], 0_K)\}$ is a $\delta$-fine   partition for $[0_K, x]$.  

Since $S$ is nonempty and bounded above by $1_K$, we may define  $y = \sup S \in K$. If $y$ is left-isolated, let $y^-$ be its predecessor. If $y^- = 0_K$, then by the argument above, we have $y \in S$. If $y^- > 0_K$, then $y^- \in S$, for otherwise, $y^-$ would be an upper bound of $S$ strictly smaller than $\sup S = y$, contradicting the definition of $y$.  
Let $P$ be a $\delta$-fine   partition of $[0_K, y^-]$. Then, defining  $P' = P \cup \{([y^-, y], y)\}$
yields a $\delta$-fine   partition of $[0_K, y]$, so that $y \in S$. If $y$ is left-dense, then there exists $x \in [0_K, y) \cap \delta(y) \cap S$. Given a $\delta$-fine   partition $P$ of $[0_K, x]$, we define $P' = P \cup \{([x, y], y)\}$. This forms a $\delta$-fine   partition of $[0_K, y]$, ensuring that $y \in S$. In all cases, we conclude that $y \in S$.  

To complete the proof, we must show that $y = 1_K$. Assume, for contradiction, that $y < 1_K$, and let $P$ be a $\delta$-fine   partition for $[0_K, y]$.  If $y$ is right-isolated, let $y^+$ be its successor. Then the partition $P' = P \cup \{([y, y^+], y^+)\}$ is a $\delta$-fine   partition for $[0_K, y^+]$, implying that $y^+ \in S$, which contradicts the definition of $y$ as the supremum of $S$.  If $y$ is right-dense, we can choose $z \in (y, 1_K] \cap \delta(y)$. Then the partition  $P' = P \cup \{([y, z], y)\}$
is a $\delta$-fine   partition for $[0_K, z]$, implying that $z \in S$, which again contradicts the definition of $y$.  Thus, we must have $y = 1_K$, completing the proof.
\end{proof}

Before concluding this subsection, we introduce some additional terminology and record a few elementary facts about $\delta$-fine partitions, which will be used implicitly throughout the text.

\begin{enumerate}[label=\textnormal{(G\alph*)}]
\item\label{it:Gauge1} It is easily verified that if $\eta$ refines $\delta$, then every $\eta$-fine partition $P$ is also $\delta$-fine; in other words, $P \ll \eta$ implies $P \ll \delta$.

\item\label{it:Gauge2} Since the topology on our compact line is Hausdorff, given a gauge $\delta$, we can refine it to a gauge $\eta$ that \emph{exposes} a point $c \in K\setminus\{0_K\}$ in the sense that $x \neq c \Rightarrow c \notin \eta(x)$. One way to construct such a gauge is by defining $\eta(x) = \delta(x) \cap [0_K, c)$ for $x < c$, $\eta(x) = \delta(x) \cap (c, 1_K]$ for $x > c$, and $\eta(c) = \delta(c)$. This construction is useful because if a gauge $\eta$ exposes a point $c \in K$, then for every $\eta$-fine partition $P = \{([x_{i-1},x_i],t_i):1\leq i \leq n\}$ of $K$, we have that $c \in (x_{i-1}, x_i]$ implies $t_i = c$. In particular, $c$ must appear as a tag in $P$. Moreover, any finite set of points can be simultaneously exposed by iteratively applying this construction to each point in the set.

\item\label{it:Gauge3} If $P = \{([x_{i-1},x_i],t_i): 1 \leq i \leq n\}$ is a   partition of $[0_K, c]$, and $Q = \{([y_{j-1},y_j],s_j): 1 \leq j \leq m\}$ is a   partition of $[c, 1_K]$, then clearly $x_n = c = y_0$. We define the \emph{merging} of $P$ and $Q$ as  
\[\widehat{P \cup Q} = \{([z_{k-1},z_k],p_k): 1 \leq k \leq m+n\},\]  
where
\[([z_{k-1}, z_k], p_k) = 
\begin{cases}
([x_{k-1}, x_k], t_k), & \text{if } k \leq n, \\
([y_{k-1-n}, y_{k-n}], s_{k-n}), & \text{if } k > n,
\end{cases}\]
which constitutes a   partition of $K$. Given a gauge $\delta$ on $K$, it is straightforward to verify that $\widehat{P \cup Q}$ is $\delta$-fine if and only if $P$ is a $\delta$-fine   partition of $[0_K, c]$ and $Q$ is a $\delta$-fine   partition of $[c, 1_K]$. It is straightforward to extend the merging operation to three or more partitions in an analogous manner.
\end{enumerate}

\subsection{Riemann sums and integrability} 

\begin{defin}[Riemann sums and integrability]\label{def:integrability}
    Let $K$ be a compact line, $P=\{([x_{i-1},x_i],t_i): 1 \leq i \leq n\}$ be a partition of $K$, and $f,G:K\to \mathbb R$.  The \emph{Riemann sum of $f$ with respect to $P$ and $G$} is defined and denoted by
\begin{eqnarray}\label{eq:riemann}
S(f,G,P)=f(0_K)G(0_K)+\sum_{i=1}^n f(t_i)(G(x_{i})-G(x_{i-1})). 
\end{eqnarray}
We say that $f$ is \emph{Kurzweil--Stiletjes integrable with respect to $G$}, or \emph{$G$-integrable} for short, when there is an $A\in \mathbb R$ such that, for each $\varepsilon>0$, there exists a gauge $\delta$ on $K$ such that 
\begin{eqnarray}\label{eq:int condition}
P\ll \delta \Rightarrow |S(f,G,P)-A|<\varepsilon. 
\end{eqnarray}
In the affirmative case, we call $A$ the \emph{Kurzweil--Stieltjes integral of $f$} with respect to $G$ (or \emph{$G$-integral of $f$}) over $K$, and denote it by $\int_K f\, dG$ or $\int_{0_K}^{1_K} f\, dG$. A gauge satisfying \eqref{eq:int condition} is called an \emph{$\varepsilon$-gauge} (with respect to a $G$-integrable function $f$). 
\end{defin}

\begin{remark}\label{rem:CompareNabla}
    The term $f(0_k)G(0_K)$ in \eqref{eq:riemann} does not appear in the classic definition of Kurzweil-Stieltjes integration on compact real intervals (see \cite[Definition 6.1.1]{monteiro2019kurzweil}). The measure-theoretic motivation for considering this term will be explored in Subsection \ref{subsec:Radon} and especially highlighted in Theorem \ref{Thm:CharacterizationGRadon}, where a clear connection with Lebesgue integrability is established. 

    The integral defined above turns out to generalize the Kurzweil-Stieltjes nabla integral on time scales (and consequently the classic Kurzweil-Stieltjes integral on compact real intervals) modulo the correcting factor $f(0_K)G(0_K)$. We postpone the proof of this fact to the Appendix. 
\end{remark}

\begin{remark}\label{Rem:EnlargePartitions}
Let $P = \{([x_{i-1}, x_i], t_i) : 1 \leq i \leq n\}$ be a tagged partition of a compact line $K$, and suppose we are given a collection of points $c_1, \ldots, c_k \in K$ that are exposed in $P$, see \ref{it:Gauge2}. It will occasionally be convenient to refine the partition $P$ so that each point $c_j$ appears as a division point of $P$, or as the tag of two consecutive components of $P$. 

To achieve this, we split any component $([x_{i-1}, x_i], t_i)$ such that $t_i = c_j$ into two components: $([x_{i-1}, c_j], c_j)$ and $([c_j, x_i], c_j)$. The resulting partition is clearly still $\delta$-fine and does not alter the Riemann sum of $f$ with respect to $P$ and $G$, since
\[f(c_j)(G(x_i) - G(x_{i-1})) = f(c_j)(G(x_i) - G(c_j)) + f(c_j)(G(c_j) - G(x_{i-1})).\]
We may also perform the reverse operation when needed, merging adjacent components with the same tag in order to maintain a partition of the form $x_0 < x_1 < \ldots < x_n$, see \ref{it:Partition2}.
\end{remark}

The following basic result is an adaptation to our integration framework of a theorem bearing the same name in the classical Kurzweil--Stieltjes theory. It provides a useful criterion for establishing the integrability of a function in situations where no candidate value for the integral is known. In particular, it will allow us to identify a broad class of integrable functions under certain assumptions on $G$.

\begin{theorem}[Bolzano-Cauchy Criterion]\label{Thm:BolzanoCauchy}
On a compact line $K$, a function $f:K\to \mathbb{R}$ is integrable with respect to some function $G:K\to \mathbb{R}$ if and only if for every $\varepsilon>0$, there exists a gauge $\delta$ such that
\[|S(f,G, P)-S(f,G,Q)|<\varepsilon\]
for every pair of $\delta$-fine partitions $P$ and $Q$ of $K$.
\end{theorem}
\begin{proof}
Let $f$ be integrable with respect to $G$, fix $\varepsilon>0$, and choose $\delta$ to be a $\frac{\varepsilon}{2}$-gauge for $f$.
Then, given that $P$ and $Q$ are $\delta$-fine partitions of $K$, we have:
\[|S(f,G,P)-S(f,G,Q)|\leq |S(f,G,P)-\int f dG|+|\int f dG-S(f,G,Q)|<\varepsilon.\]

On the other hand, for each $n$, let $\delta_n$ be a gauge on $K$ such that 
\[P,Q\ll \delta_n\Rightarrow |S(f,G,P)-S(f,G,Q)|<1/n.\]
    
 We can assume that $ \delta_{n+1}$ refines $\delta_n$, for each $n\in \mathbb{N}$. Select a partition $P_n\ll \delta_n$ for each $n$, and observe that, for $m>n$, we have
    \begin{eqnarray*}
    |S(f,G,P_n)-S(f,G,P_m)|< 1/n.
    \end{eqnarray*}
  Thus $(S(f,G,P_n))_n$ comprises a Cauchy sequence in $\mathbb{R}$ and hence converges to some real number $A$. By passing to the limit as $m\to\infty$ on the previous relation, we obtain 
  \[|S(f,P_n,G)-A|\leq 1/n,\]
  for all $n$. To complete the proof, let $\varepsilon>0$ be arbitrary and fix a natural number $n>2/\varepsilon$. Then, for each $P\ll \delta_n$, one has
\begin{align*}
    |S(f,G,P)-A| &\leq |S(f,G,P)-S(f,G,P_n)| - |S(f,G,P_n)-A|\\
    &\leq 1/n + 1/n <\varepsilon, 
\end{align*}
which concludes the proof.  
\end{proof}

\begin{remark}\label{Rem:EpsilonGauge2}
Taking in account Theorem \ref{Thm:BolzanoCauchy} and aiming the simplification of proofs, we can strengthen the definition of $\varepsilon$-gauge as follows. If $f$ is integrable over a compact line $K$ with respect to some function $G$, for every $\varepsilon > 0$, when we say that  $\delta$ is an $\varepsilon$-gauge for $f$, we will assume that
\begin{enumerate}
    \item $\left|S(f, G, P)-\int_K f\,dG\right|<\varepsilon$, provided that $P$ is a $\delta$-fine   partition of $K$, and
    \item $|S(f, G, P) - S(f, G, Q)| < \varepsilon$ whenever $P, Q \ll \delta$.
\end{enumerate} 
\end{remark}

\begin{lemma}\label{Lem:UniformContinuity}
Let $K \subset \mathbb{R}$ be a compact line, and let $f: K \to \mathbb{R}$ be a continuous function. Then, for every $\varepsilon > 0$, there exists a division $\{ z_0, z_1, \ldots, z_n\}$ of $K$ such that for each $i = 1, \dots, n$, either:
\begin{itemize}
    \item $z_i$ is left-isolated, or
    \item for all $x, y \in [z_{i-1}, z_i]$, we have $|f(x) - f(y)| < \varepsilon$.
\end{itemize}
\end{lemma}
\begin{proof}
Let $\varepsilon > 0$ be arbitrary. For each $w \in K$, consider a closed interval (in the sense that contains its endpoints) $J_w \subseteq f^{-1}\left((f(w) - \frac{\varepsilon}{2}, f(w) + \frac{\varepsilon}{2})\right)$ that is also a neighborhood of $w$. The collection  $\mathcal{C} = \{ J_w : w \in K \}$ forms an open cover of $K$ when considering the interiors of the intervals $J_w$. Since $K$ is compact, there exist finitely many points $w_1 < w_2 < \cdots < w_m$ such that  $K \subseteq J_{w_1} \cup J_{w_2} \cup \cdots \cup J_{w_m}$.
We may assume that $m$ is the smallest integer for which such a sequence exists.

We use the endpoints of the intervals $J_{w_1}, \dots, J_{w_m}$ to define a division $\{z_0, z_1, \ldots, z_n\}$ of $K$. Let $1 \leq i \leq n$ be arbitrary. If $z_i$ is left-dense, let $1 \leq k \leq m$ be the smallest integer such that $z_i \in J_{w_k}$. If $z_i = 0_{J_{w_k}}$, then $z_{i-1} = z_i^-$, implying that $z_i$ is left-isolated, which contradicts the assumption. Therefore, the interval 
$[z_{i-1}, z_i] \subset J_{w_k} \subset f^{-1}\left( (f(w_k) - \frac{\varepsilon}{2}, f(w_k) + \frac{\varepsilon}{2}) \right)$. Now, for any $x, y \in [z_{i-1}, z_i]$, we have:
\[|f(x) - f(y)| \leq |f(x) - f(w_i)| + |f(w_i) - f(y)| < \frac{\varepsilon}{2} + \frac{\varepsilon}{2} = \varepsilon.\]
This completes the proof.
\end{proof}

Following the notation of \cite{ronchim2021study} that extends the corresponding notion for time scales, we say that on a compact line $K$, a function $G:K\to \mathbb{R}$ is said to have \emph{bounded variation} if its \emph{total variation},
\[\operatorname{Var}(G, K) := \sup \left\{ \sum_{i=1}^n |G(x_i) - G(x_{i-1})| : \{x_0, x_1, \ldots, x_n\} \text{ is a division of } K \right\},\]
is finite. When the domain of $G$ is clear from the context, we will simply write $\operatorname{Var}(G)$.

The following theorem provides a rich class of $G$-integrable functions, assuming that $G$ is of bounded variation.

\begin{theorem}\label{Thm:ContinuousIsGintegrable}
Let $K$ be a compact line, and let $G: K \to \mathbb{R}$ be a function of bounded variation. Then every continuous function 
$f: K \to \mathbb{R}$ is $G$-integrable. Moreover, the following inequality holds:
\[\left|\int_K f \, dG\right| \leq  |f(0_K) G(0_K)|+\sup_{x \in K} |f(x)| \,\operatorname{Var}(G).\]
\end{theorem}
\begin{proof}
Let $\varepsilon > 0$ be arbitrary. Since $f$ is continuous, by Lemma~\ref{Lem:UniformContinuity}, we can fix a division 
$\mathcal{C} = \{z_0, z_1, \ldots, z_r\}$ such that, for each $1 \leq i \leq r$, either $z_i$ is left-isolated or for all  $x, y \in [z_{i-1}, z_i]$, we have 
$|f(x) - f(y)| < \frac{\varepsilon}{\operatorname{Var}(G) + 1}$. Let $\delta$ be a gauge for $K$ that exposes all the points in $\mathcal{C}\setminus \{0_K\}$ and moreover, whenever $z\in \mathcal{C}\setminus\{1_K\}$ is a left-isolated point, 
$\delta(z)=[z,v)$ for some $v\in K$. Let $P, Q$ be $\delta$-fine partitions of $K$. By refining the partitions $P$ and $Q$ if necessary, we may assume that every element of $\mathcal{C} \setminus \{0_K\}$ appears both as a tag and as a division point in each of the partitions $P$ and $Q$ (see Remark~\ref{Rem:EnlargePartitions}).

Assume $P = \{(I_i, t_i): 1 \leq i \leq l\}$ and $Q = \{(J_j, s_j): 1 \leq j \leq m\}$. Let the union of the endpoints of all intervals $I_i$ and $J_j$ form the set $\{x_0, x_1, \ldots, x_n\}$. Define new partitions 
$P^* = \{([x_{k-1}, x_k], t'_k): 1 \leq k \leq n\}$ and $Q^* = \{([x_{k-1}, x_k], s'_k): 1 \leq k \leq n\}$ as follows: 

For each $1 \leq i \leq l$, if $I_i = \bigcup_{k=k_0}^{k_1} [x_{k-1}, x_k]$, set $t'_k = t_i$ for $k_0 \leq k \leq k_1$. Similarly, for each $1 \leq j \leq m$, if $J_j = \bigcup_{k=k_0}^{k_1} [x_{k-1}, x_k]$, set $s'_k = s_j$ for $k_0 \leq k \leq k_1$.

Now observe that if $I_i = [a, b]$ and $I_i = \bigcup_{k=k_0}^{k_1} [x_{k-1}, x_k]$, we have
\[
f(t_k) (G(b) - G(a)) = \sum_{k=k_0}^{k_1} f(t'_k) (G(x_k) - G(x_{k-1})),
\]
implying $S(f, G, P) = f(0_K) G(0_K) + \sum_{k=1}^n f(t'_k) (G(x_k) - G(x_{k-1})) = S(f, G, P^*)$. The same holds for $Q$, leading to:
\[
|S(f, G, P) - S(f, G, Q)| = |S(f, G, P^*) - S(f, G, Q^*)| = \left| \sum_{k=1}^n (f(t'_k) - f(s'_k)) (G(x_k) - G(x_{k-1})) \right|.
\]
Let $1 \leq k \leq n$ be arbitrary and consider the   intervals $([x_{k-1},x_k],t'_k)$ and $([x_{k-1},x_k],s'_k)$. If $x_k \in \mathcal{C}$ is a left-isolated point, by the construction of the gauge $\delta$ we necessarily have $t'_k=s'_k=x_k$ and  $x_{k-1} = x_k^-$. Consequently, $(f(t'_k) - f(s'_k)) (G(x_k) - G(x_{k-1}))=0$. Otherwise, there must exist $1 \leq j \leq r$ such that
$[x_{i-1},x_i]\subset [z_{j-1}, z_j]$. Since  $t'_k, s'_k \in [x_{i-1},x_i]$ we have
\begin{align*}
|S(f, G, P) - S(f, G, Q)| &\leq \sum_{k=1}^n |f(t'_k) - f(s'_k)| \, |G(x_k) - G(x_{k-1})| \\
&< \frac{\varepsilon}{\operatorname{Var}(G) + 1} \sum_{k=1}^n |G(x_k) - G(x_{k-1})| < \varepsilon.
\end{align*}
By the Bolzano-Cauchy Criterion \ref{Thm:BolzanoCauchy}, $f$ is $G$-integrable.

For the second part, let $\varepsilon > 0$ and let $\delta$ be an $\varepsilon$-gauge for $f$. Given a $\delta$-fine partition 
$P = \{([x_{i-1},x_i],t_i):1\leq i \leq n\}$ of $K$, we have
\begin{align*}
\left|\int_K f \, dG\right| &\leq |S(f, G, P)| + \varepsilon\leq |f(0_K) G(0_K)| + \sum_{i=1}^n |f(t_i)| \, |G(x_i) - G(x_{i-1})| + \varepsilon \\
&\leq  |f(0_K) G(0_K)|+\sup_{x \in K} |f(x)| \,\operatorname{Var}(G) + \varepsilon.
\end{align*}
Since $\varepsilon$ is arbitrary, the result follows.
\end{proof}

%\begin{remark}\label{rem:DualidadeDaParticao}
%    Suppose that  $P=\{x_0\leq t_1\leq x_1\leq\dots\leq t_n\leq x_n\}$ is a $\delta$-fine partition of $K$, and $\varepsilon>0$. It is immediately seen that if $\delta$ is an $\varepsilon$-gauge and $\eta$ is a refinement of $\delta$, then $\eta$ is an $\varepsilon$-gauge as well.  It is sometimes practical to refine $\delta$ to a gauge $\eta$ that exposes a point of interest $c\in (0_K,1_K)$. When $P=\{x_0\leq t_1\leq x_1\leq\dots\leq t_n\leq x_n\}\ll \eta$ for such $\eta$, recall that $c\in (x_{i-1},x_i]\Rightarrow c=t_i$. There are two alternatives for $P$: either $c\in (x_{j-1},x_j)$ for some $j\in \{1,\dots,n\}$, or $c=t_{j-1}=t_j$ for some $j\in\{2,\dots,n\}$. It turns out that we can assume one that one of these scenarios holds according to our convenience without affecting the Riemann sum. In effect, if $c=t_{j-1}=t_j$ for some $j\in\{2,\dots,n\}$, 
 %   \begin{align*}
 %   &f(t_{j-1})(G(x_{j-1})-G(x_{j-2}))+f(t_{j})(G(x_{j})-G(x_{j-1}))\\ 
 %   =& f(c)(G(c)-G(x_{j-2}))+f(c)(G(x_{j})-G(c))\\
 %   =&f(c)(G(x_j)-G(x_{j-2})), 
  %  \end{align*}
  %  thus $S(f,G,P)=S(f,P',G)$, where $$P'=\{x_0\leq t_1\leq x_1\leq \dots x_{j-2}\leq c \leq x_j\leq\dots\leq t_n\leq x_n\}$$ 
  %  and $x_{j-2}<c<x_j$. 
%\end{remark}}

\subsection{Regularity of integrators}

As established in the previous section (Theorem~\ref{Thm:ContinuousIsGintegrable}), if a function $G$ defined on a compact line has bounded variation, then all continuous functions are $G$-integrable. However, to guarantee desirable properties of the $G$-integral, further assumptions on the function $G$ are required. 

Although the definition of the $G$-integral can, in principle, be formulated with respect to any function defined on a compact line, our investigation focuses primarily on three levels of regularity for $G$. Before presenting and analyzing these levels, we introduce some additional terminology.

We say that $G$ is \emph{right-continuous at} a point $x \in [0_K, 1_K)$ if, for every $\varepsilon > 0$, there exists a neighborhood $V$ of $x$ such that for all $y \in V \cap [x, 1_K]$, we have $|G(y) - G(x)| < \varepsilon$. The function $G$ is said to be \emph{right-continuous} if it is right-continuous at every point $x \in [0_K, 1_K)$.

If $x \in [0_K, 1_K)$ is a right-dense point, we say that the \emph{right-hand limit of $G$ at $x$ exists and is equal to} $A \in \mathbb{R}$ if, for every $\varepsilon > 0$, there exists a neighborhood $V$ of $x$ such that for all $y \in V \cap (x, 1_K]$, we have $|G(y) - A| < \varepsilon$. In this case, we write $A = \lim_{y \searrow x} G(y)$. Left-hand continuity and left-hand limits are defined analogously.

We say that a function $G: K \to \mathbb{R}$ is \emph{regulated} at a point $x \in K$ if the following conditions hold: either $x$ is left-isolated in $K$, or $x$ is left-dense in $K$ and the left-hand limit of $G$ exists at $x$; and either $x$ is right-isolated in $K$, or $x$ is right-dense in $K$ and the right-hand limit of $G$ exists at $x$. If $G$ is regulated at every point of $K$, we say that $G$ is a \emph{regulated function} on $K$. 

If $G$ is regulated at a point $x \in K$, we define
\[
L_G(x) = 
\begin{cases}
0 & \text{if } x = 0_K, \\
\lim_{y \nearrow x} G(y) & \text{if } x \in (0_K, 1_K] \text{ is left-dense}, \\
G(x^-) & \text{if } x \in (0_K, 1_K] \text{ is left-isolated}.
\end{cases}
\]
In particular, if $G$ is regulated on $K$, the function $L_G: K \to \mathbb{R}$ is well defined at each $x \in K$ by the formula above. This auxiliary function encapsulates the left-limit behavior of $G$ in a unified way and will serve as a fundamental tool in our analysis, especially in the formulation and interpretation of integrals with respect to $G$.

The three levels of regularity we focus on are as follows.

\
 
\paragraph{\textbf{Regularity Level 1: $G$ is amenable}} We say that $G$ is \emph{amenable} if it is both right-continuous and regulated. These conditions ensure that the integral of characteristic functions of intervals can be properly computed, and they also guarantee a satisfactory additivity property of the integral on subintervals. Moreover, under these assumptions, it is possible to establish a version of the Saks-Henstock lemma, which plays a central role in the theoretical framework.

\

\paragraph{\textbf{Regularity Level 2: $G$ is amenable and has bounded variation}}  
Adopting the terminology from \cite{ronchim2021study}, we denote by $\mathrm{NBV}(K)$ the space of functions that are both amenable and of bounded variation. This space becomes a Banach space when equipped with the norm $\|G\| = |G(0_K)| + \operatorname{Var}(G)$. 

When $G \in \mathrm{NBV}(K)$, we can rely on a powerful result from \cite[Theorem 2.4.11]{ronchim2021study} (see also \cite{tausk2023}), which states that $\mathrm{NBV}(K)$ can be naturally and isometrically identified with the space of Radon measures on $K$, denoted by $\mathcal{M}(K)$, endowed with the total variation norm $\|\mu\| = |\mu|(K)$. More precisely, the linear mapping that assigns to each measure $\mu \in \mathcal{M}(K)$ the function $G_\mu \in \mathrm{NBV}(K)$, defined by $G_\mu(x) = \mu([0_K, x])$ for every $x \in K$, is a well-defined isometric isomorphism.

At this level of regularity, we are able to define the concept of $G$-null sets. This notion allows us to identify functions that coincide $G$-almost everywhere, in a suitable sense, with respect to their integrability.

\

\paragraph{\textbf{Regularity Level 3: $G$ is amenable, positive and nondecreasing}}  
In this case, we observe in particular that $G \in \mathrm{NBV}(K)$, since every nondecreasing function has bounded variation. Under this stronger regularity assumption, we are able to introduce the notion of $G$-differentiability and establish a version of the Fundamental Theorem of Calculus, which connects the $G$-integral with the $G$-derivative in a precise and meaningful way. At this level of regularity, we can also prove the standard convergence theorems, such as the Monotone Convergence Theorem and the Dominated Convergence Theorem.

\subsection{Foundational properties and a Saks-Henstock-type lemma}
\label{Sec-BasicProperties}
We begin this section by establishing the fundamental and important fact that the integral defines a bilinear form. Indeed, given a   partition $P$ of a compact line $K$, it is straightforward to verify that the mapping $(f, G) \mapsto S(f, G, P)$ defines a bilinear form on the vector space of real-valued functions on $K$. From this observation, the bilinearity of the integral follows naturally, as stated in the next proposition.

\begin{prop}\label{Prop:BilinearityOfHKIntegral}
Let $K$ be a compact line and $G: K \to \mathbb{R}$ be a fixed function. If $f_1, f_2: K \to \mathbb{R}$ are $G$-integrable, and $\lambda$ is a real scalar, then $\lambda f_1 + f_2$ is integrable,  and satisfies
\[\int_K (\lambda f_1 + f_2) \, dG = \lambda \int_K f_1 \, dG + \int_K f_2 \, dG.\]
Furthermore, if a function $ f: K \to \mathbb{R} $ is integrable with respect to functions $G_1$ and $G_2$, then for every real scalar $\lambda$, $f$ is integrable with respect to $\lambda G_1 + G_2$, and we have
\[\int_K f \, d(\lambda G_1 + G_2) = \lambda \int_K f \, dG_1 + \int_K f \, dG_2.\]
\end{prop}
\begin{proof}
Let $ A_1 = \int_K f_1 \, dG $ and $ A_2 = \int_K f_2 \, dG $. For a given $ \varepsilon > 0 $, let $ \delta $ be a $ \frac{\varepsilon}{2(1 + |\lambda|)} $-gauge for both $ f_1 $ and $ f_2 $ %!!!(see Fact 1 of Proposition \ref{Prop:PartitionFacts}). 
Then, for any partition $ P \ll \delta $, we have
\[
\left| S(\lambda f_1 + f_2, P, G) - \lambda A_1 - A_2 \right| \leq |\lambda| \left| S(f_1, P, G) - A_1 \right| + \left| S(f_2, P, G) - A_2 \right| < \varepsilon.
\]
This shows that $ \lambda f_1 + f_2 $ is integrable and that the integral satisfies the desired linearity property. The second part can be proved using a similar argument with an appropriate gauge.
\end{proof}

To establish the additivity property on subintervals (Theorem~\ref{Thm:Additivity} below), we begin by explicitly computing the integral of the characteristic function of a singleton. In compact lines, such points often carry nonzero mass, which leads to a correction term in the additivity formula. Assuming $G$ is amenable, this computation is manageable, as the following proposition shows.

\begin{prop}\label{Prop:int singleton}
    Let $G$ be an amenable function on a compact line $K$, and let $c \in K$. Then $\chi_{\{c\}}$ is $G$-integrable on $K$, and 
    \begin{equation}\label{eq:int singleton}
        \int_K \chi_{\{c\}}\,dG = G(c) - L_G(c). 
    \end{equation}
\end{prop}    

\begin{proof}
    We prove the result for $c \in (0_K, 1_K)$; the remaining cases follow by a simplified version of the same argument. Let $\varepsilon > 0$ be arbitrary. Since $G$ is amenable, there exists an open interval $I$ containing $c$ such that:
    \begin{itemize}
        \item $x \in I \cap [0_K, c) \Rightarrow |G(x) - L_G(c)| < \frac{\varepsilon}{2}$,
        \item $x \in I \cap [c, 1_K] \Rightarrow |G(x) - G(c)| < \frac{\varepsilon}{2}$.
    \end{itemize}
We define a gauge $\delta$ on $K$ by setting $\delta(x) = [0_K, c)$ for $x < c$ and $\delta(x) = (c, 1_K]$ for $x > c$. At the point $c$, if it is left-isolated, we define $\delta(c) = I \cap [c, 1_K]$; if $c$ is left-dense, we select a point $y_c < c$ such that $[y_c, c] \subset I$ and let $\delta(c) = I \cap (y_c, 1_K]$. Finally, we refine $\delta$ to ensure that the point $c$ is exposed in any $\delta$-fine partition.

Let $P = \{([x_{i-1}, x_i], t_i): 1 \leq i \leq n\}$ be a $\delta$-fine   partition of $K$. By refining the partition if necessary, we may assume that $c$ appears both as a tag and as a division point of two consecutive components of $P$ (see Remark~\ref{Rem:EnlargePartitions}). Let $1 \leq i \leq n$ be such that $t_i = c = t_{i+1}$. Then, the Riemann sum becomes:
\begin{align*}
    S(\chi_{\{c\}}, G, P) &= \chi_{\{c\}}(c)\big(G(c) - G(x_{i-1})\big) + \chi_{\{c\}}(c)\big(G(x_{i+1}) - G(c)\big)\\
    &= G(x_{i+1}) - G(x_{i-1}).
\end{align*}
    Therefore,
    \begin{align*}
        \left|S(\chi_{\{c\}}, G, P) - \big(G(c) - L_G(c)\big)\right|
        &= \left|G(x_{i+1}) - G(x_{i-1}) - G(c) + L_G(c)\right| \\
        &\leq |G(x_{i+1}) - G(c)| + |L_G(c) - G(x_{i-1})|.
    \end{align*}
    Since $x_{i+1} \in I \cap [c, 1_K]$, we have $|G(x_{i+1}) - G(c)| < \frac{\varepsilon}{2}$. If $c$ is left-isolated, then $x_{i-1} = c^-$ and $|L_G(c) - G(x_{i-1})| = 0$. Otherwise, $x_{i-1} \in I \cap [0_K, c)$, so $|L_G(c) - G(x_{i-1})| < \frac{\varepsilon}{2}$.

    Thus,
    \[
        \left|S(\chi_{\{c\}}, G, P) - \big(G(c) - L_G(c)\big)\right| < \varepsilon,
    \]
    and the proof is complete.
\end{proof}

Our next step will be to establish the $G$-integrability of indicator functions determined by a $G$-integrable function $f$.

\begin{defin}\label{def:IndicatorFunction}
Let $K$ be a compact line and $I\subset K$ be a subinterval. For every function $f:K\to \mathbb{R}$, the indicator function of $I$ with respect to $f$ is the function $f_I:K\to \mathbb{R}$ defined by the formula
\[ f_I(x) = \left\{ \begin{array}{cc}
f(x) & \text{ if } x \in I\\
0 & \text{ otherwise.}  
\end{array} \right. \]
\end{defin}

\begin{prop}\label{Prop:IndicatorFunctionIntegral}
Let $K$ be a compact line and let $G:K\to \mathbb{R}$ be an amenable function. If a function $f:K\to \mathbb{R}$ is $G$-integrable over $K$, then the indicator function $f_I:K\to \mathbb{R}$ is $G$-integrable over $K$ for every nonempty subinterval $I \subset K$.
\end{prop}
\begin{proof}
By the linearity of the integral and Proposition~\ref{Prop:int singleton}, it suffices to prove the statement for $I = [0_K, b]$ with $0_K < b < 1_K$. Let $\varepsilon > 0$ be arbitrary. Since $f: K \to \mathbb{R}$ is $G$-integrable over $K$, there exists a $\frac{\varepsilon}{2}$-gauge $\delta$ for $f$ on $K$. Furthermore, as $G$ is amenable, we may assume that 
\[
|G(x) - G(y)| < \frac{\varepsilon}{2(|f(b)| + 1)}
\quad \text{for all } x, y \in \delta(b) \cap [b, 1_K].
\]

By Lemma~\ref{lem:cousin}, fix a $\delta$-fine   partition $R$ of $[b, 1_K]$. Let $P$ and $Q$ be arbitrary $\delta$-fine   partitions of $K$. By refining the partitions if necessary, we may assume that $b$ appears both as a tag and as a division of two consecutive components in each of the partitions $P$ and $Q$ (see Remark~\ref{Rem:EnlargePartitions}). Define $P_I$ and $Q_I$ as the subpartitions of $P$ and $Q$ with intervals contained in $I$ and note that the merged partitions $P_1 = \widehat{P_I \cup R}$ and $Q_1 = \widehat{Q_I \cup R}$ are also $\delta$-fine  partitions of $K$.

Let $([x_{i-1}, x_i], t_i)$ and $([y_{j-1}, y_j], s_j)$ be the tagged components of $P$ and $Q$, respectively, such that $x_{i-1} = b = t_i$ and $y_{j-1} = b = s_j$. By construction, we have $x_{i}, y_{j} \in \delta(b) \cap [b, 1_K]$, and thus:
\begin{align*}
|S(f_I, G, P) - S(f_I, G, Q)| 
&= \Big| S(f, G, P_1) - S(f, G, Q_1) \\
&\quad + f(b)\big(G(x_{i}) - G(b)\big) - f(b)\big(G(y_{j}) + G(b)\big) \Big| \\
&\leq |S(f, G, P_1) - S(f, G, Q_1)| + |f(b)||G(x_{i}) - G(y_{j})| \\
&\leq \frac{\varepsilon}{2} + \frac{\varepsilon}{2} = \varepsilon.
\end{align*}

We conclude that $|S(f_I, G, P) - S(f_I, G, Q)| < \varepsilon$. By Theorem~\ref{Thm:BolzanoCauchy}, it follows that $f_I$ is $G$-integrable over $K$.
\end{proof}

Given a subinterval $I$ of a compact set $K$, and a function $f: K \to \mathbb{R}$, in addition to the indicator function of $I$ with respect to $f$, as described above, we may also consider the restriction of $f$ to $I$, denoted $f|_I: I \to \mathbb{R}$, where $f|_I(x) = f(x)$ for every $x \in I$. These functions are naturally related, as are their integrals. The next proposition will establish the relation between these concepts.

For simplicity, throughout this text, we will say that $f|_I$ is \emph{$G$-integrable} rather than explicitly stating that $f|_I$ is $G|_I$-integrable. Whenever the integral exists, we will denote it by $\int_I f \, dG$, omitting the restriction notation $f|_I$.
 
\begin{prop}\label{prop:IntegrationOnASubinterval}
Let $K$ be a compact line and let $G:K\to \mathbb{R}$ be an amenable function. For every function $f:K\to \mathbb{R}$ and every compact subinterval $I\subset K$, the function $f_I$ is $G$-integrable on $K$ if and only if $f|_I:I\to \mathbb{R}$ is $G$-integrable on $I$. Moreover, we have:
\[
\int_K f_I\,dG = \int_I f\,dG - f(0_I)L_G(0_I).
\]
\end{prop}
\begin{proof}
Let $I = [a, b] \subseteq K$ be an arbitrary closed subinterval. We assume that $0_K < a < b < 1_K$, since the remaining cases follow by simpler versions of the same argument.

Assume first that $f_I$ is $G$-integrable on $K$, and fix $\varepsilon > 0$. Then we may choose an $\varepsilon$-gauge $\delta$ for $f_I$ on $K$. By Lemma~\ref{lem:cousin}, we can fix $\delta$-fine partitions $R_a$ of $[0_K, a]$ and $R_b$ of $[b, 1_K]$.

Now, let $P$ and $Q$ be arbitrary $\delta$-fine partitions of $I$. Define $P_1 = \widehat{R_a \cup P \cup R_b}$ and $Q_1 = \widehat{R_a \cup Q \cup R_b}$. These are $\delta$-fine partitions of $K$, and we have:
\[
|S(f|_I, G, P) - S(f|_I, G, Q)| = |S(f_I, G, P_1) - S(f_I, G, Q_1)| < \varepsilon.
\]
Hence, by Theorem~\ref{Thm:BolzanoCauchy}, $f|_I$ is $G$-integrable on $I$.

Conversely, suppose that $f|_I$ is $G$-integrable on $I$. Let $\delta$ be an $\frac{\varepsilon}{2}$-gauge for $f|_I$. Let $a_1$ be the right endpoint of $\delta(a)$, and $b_0$ the left endpoint of $\delta(b)$. Since $G$ is amenable, there exists $b_1 > b$ such that $|G(x) - G(y)| < \frac{\varepsilon}{4(|f(b)| + 1)}$ for all $x, y \in [b, b_1)$. 

If $a$ is left-isolated, we define $a_0 = a^-$. If $a$ is left-dense, by the amenability of $G$, we may choose $a_0 < a$ such that $|G(x) - G(y)| < \frac{\varepsilon}{4(|f(a)| + 1)}$ for all $x, y \in [a_0, a)$.

Define the gauge $\delta_1$ on $K$ by:
\[
\delta_1(t) =
\begin{cases}
    [0_K, a) & \text{if } t < a, \\
    (a_0, a_1) & \text{if } t = a, \\
    \delta(t) & \text{if } t \in (a, b), \\
    (b_0, b_1) & \text{if } t = b, \\
    (b, 1_K] & \text{if } t > b.
\end{cases}
\]
We further refine $\delta_1$ to ensure that $a$ and $b$ are exposed in any $\delta_1$-fine partition of $K$. Let $P$ and $Q$ be such partitions. By inserting additional points and tags, we may assume that $a$ and $b$ appear as both tags and division points in two consecutive components of each partition.

Let $P_I$ and $Q_I$ be the subpartitions of $P$ and $Q$ consisting of components entirely contained in $I$. Then $P_I$ and $Q_I$ are $\delta$-fine partitions of $I$. Let $([x_{i-1},x_i],t_i)$ and $([y_{j-1},y_j],s_j)$ be the components of $P$ and $Q$ such that $x_i = a = t_i$ and $y_j = a = s_j$, and let $([x_{p-1},x_p],t_p)$ and $([y_{q-1},y_q],s_q)$ be those such that $x_{p-1} = b = t_p$ and $y_{q-1} = b = s_q$.

Then, from the construction of the gauge $\delta_1$ we have:
\begin{align*}
|S(f_I, G, P) - S(f_I, G, Q)| &= \big|S(f|_I, G, P_I) - S(f|_I, G, Q_I) \\
&+f(a)\big(G(a) - G(x_{i-1})\big)- f(a)\big(G(a)-G(y_{j-1})\big)\\
&+f(b)\big(G(x_{p}) - G(b)\big)- f(b)\big(G(y_{q}) - G(b)\big)\big|\\
&<\frac{\varepsilon}{2}+ |f(a)||G(x_{i-1}) - G(y_{j-1})| + |f(b)||G(x_p) - G(y_q)| \\
&< \frac{\varepsilon}{2} + \frac{\varepsilon}{4} + \frac{\varepsilon}{4} = \varepsilon.
\end{align*}
By Theorem~\ref{Thm:BolzanoCauchy}, $f_I$ is $G$-integrable over $K$.

\ 

To prove the remainder statement, fix $\varepsilon > 0$. By the amenability of $G$, choose $v > b$ such that $|G(x) - G(y)| < \frac{\varepsilon}{6(1 + |f(b)|)}$ for all $x, y \in [b, v)$. If $a$ is left-isolated, set $u = a^-$; otherwise, choose $u < a$ so that $|G(x) - G(y)| < \frac{\varepsilon}{6(1 + |f(a)|)}$ for all $x, y \in [u, a)$.

Let $\delta$ be a $\frac{\varepsilon}{3}$-gauge for $f_I$ on $K$ such that $a$ and $b$ are exposed, and such that the induced gauge $\delta_I$ is a $\frac{\varepsilon}{3}$-gauge for $f|_I$. Further refine $\delta$ so that $\delta(a) \cap [0_K, a] \subset (u, a]$ and $\delta(b) \cap [b, 1_K] \subset [b, v)$.

Let $P = \{([x_{i-1},x_i],t_i):1\leq i \leq n \}$ be a $\delta$-fine partition of $K$, where $a$ and $b$ appear as both tag and division points of two consecutive components. Let $1\leq s < r \leq n$ such that $t_s = a = t_{s+1}$ and $t_r = b = t_{r+1}$. Let $P_I$ be the partition of $I$ consisting of the components of $P$ that have interval contained in $I$.

We have:
\begin{align*}
S(f_I,G,P)&=-f(a)G(a)+f(a)(G(a)-G(x_{s-1}))+f(b)(G(x_{r+1})-G(b))+S(f|_I,G,P_I)\\
&=-f(a)L_G(a)+f(a)(L_G(a)-G(x_{s-1}))+f(b)(G(x_{r+1})-G(b))+S(f|_I,G,P_I)
\end{align*}
and from the construction of the gauge we obtain
\begin{align*}
|S(f_I,P,G)-f(a)L_G(a)-S(f|_I,G,P_I)|&\leq |f(b)||G(x_r)-G(b)|\\
&+|f(a)||L_G(a)-G(x_{s-1})|<\frac{\varepsilon}{3}.
\end{align*}
Therefore
\begin{align*}
\left|\int_K f_I\,dG-L_G(a)-\int_I f\,dG \right|&\leq \left|\int_K f_I\,dG-S(f_I,G,P) \right|\\
&+\left|S(f_I,G,P)-L_G(a)-S(f|_I,G,P_I) \right|\\
&+\left|S(f|_I,G,P_I)-\int_I f\,dG \right|<\frac{\varepsilon}{3}+\frac{\varepsilon}{3}+\frac{\varepsilon}{3}=\varepsilon.
\end{align*}
\end{proof}

We are now in a position to establish a theorem concerning the additivity of the integral with respect to subintervals.

\begin{theorem}[Additivity on subintervals]\label{Thm:Additivity}
Let $[a,b]$ be a compact line and let $G:[a,b]\to \mathbb{R}$ be an amenable function. For every $c\in [a,b]$, a function $f$ is $G$-integrable on $[a,b]$ if and only if it is $G$-integrable on both $[a,c]$ and $[c,b]$. In that case, we have:
\[\int_a^b f\,dG = \int_a^c f\,dG + \int_c^b f\,dG - f(c)G(c).\]
\end{theorem}
\begin{proof}
Let $I = [a, c]$ and $J = [c, b]$. Then we may write $f = f_I + f_J - f_{\{c\}}$. By Proposition~\ref{Prop:IndicatorFunctionIntegral} and the linearity of the integral, it follows that $f$ is $G$-integrable on $[a,b]$ if and only if $f_I$, $f_J$, and $f_{\{c\}}$ are $G$-integrable. 

Applying Proposition~\ref{prop:IntegrationOnASubinterval}, we obtain:
\begin{align*}
\int_a^b f\,dG &= \int_a^b f_I\,dG + \int_a^b f_J\,dG - \int_a^b f_{\{c\}}\,dG \\
&= \int_a^c f\,dG + \int_c^b f\,dG - f(c)L_G(c) - f(c)(G(c) - L_G(c)) \\
&= \int_a^c f\,dG + \int_c^b f\,dG - f(c)G(c). \qedhere
\end{align*}
\end{proof}

\begin{defin}
Let $K$ be a compact line. A \emph{tagged system of intervals} in $K$ is a finite collection $\mathcal{S} = \{([a_k,b_k], t_k) : k = 1, \ldots, m\}$, where $t_k \in [a_k, b_k]$ for all $1 \leq k \leq m$, and the intervals are ordered so that $b_k \leq a_{k+1}$ for each $1 \leq k < m$. Moreover, given a gauge $\delta$ on $K$, we say that the tagged system $\mathcal{S}$ is \emph{$\delta$-fine} if for every $1\leq k \leq m$ we have that $(a_k, b_k]\subset \delta(t_k)$.
\end{defin}

\begin{lemma}[Saks-Henstock]\label{lem:SH}
Let $K$ be a compact line, and let $G: K \to \mathbb{R}$ be an amenable function. If a function $f: K \to \mathbb{R}$ is $G$-integrable, then for every $\varepsilon > 0$, and every $\varepsilon$-gauge $\delta$ for $f$, the following holds: for every $\delta$-fine tagged system of intervals $\mathcal{S} = \{([a_k, b_k], t_k) : k = 1, \ldots, m\}$ in $K$, we have
\begin{equation}\label{eq:SH_condition}
\left| \sum_{k=1}^m \left( f(t_k)(G(b_k) - G(a_k)) + f(a_k)G(a_k) - \int_{a_k}^{b_k} f\, dG \right) \right| \leq 2\varepsilon.
\end{equation}
\end{lemma}
\begin{proof}
Given an arbitrary $\varepsilon > 0$, let $\delta$ be an $\varepsilon$-gauge for the function $f$, and let $\mathcal{S} = \{([a_k, b_k], t_k) : k = 1, \ldots, m\}$ be a $\delta$-fine tagged system of $K$. We set $b_0 = 0_K$ and $a_{m+1} = 1_K$. Since $f$ is $G$-integrable, it follows from Propositions \ref{Prop:IndicatorFunctionIntegral} and \ref{prop:IntegrationOnASubinterval} that the restriction $f|_{[b_k, a_{k+1}]}$ is $G$-integrable over $[b_k, a_{k+1}]$ for each $0 \leq k \leq m$.

If $\delta_k$ denotes the gauge on $[b_k, a_{k+1}]$ defined by $\delta_k(x) = \delta(x) \cap [b_k, a_{k+1}]$, we may choose $\eta_k \ll \delta_k$ to be a $\frac{\varepsilon}{m}$-gauge for $f|_{[b_k, a_{k+1}]}$. By applying Lemma \ref{lem:cousin}, we can fix an $\eta_k$-fine partition $Q_k$ of each subinterval $[b_k, a_{k+1}]$. 

On the other hand, by hypothesis, for each $1 \leq k \leq m$, the set $P_k = \{([a_k, b_k], t_k)\}$ forms a $\delta$-fine partition of $[a_k, b_k]$. We then form a $\delta$-fine partition $P^{\prime}$ of $K$ by merging all the partitions $Q_0, P_1, Q_1, \ldots, P_m, Q_m$, and write
\[\sum_{k=1}^m f(t_k)(G(b_k) - G(a_k)) = S(f, G, P^{\prime}) - \sum_{k=0}^{m} S(f|_{[b_k, a_{k+1}]}, G, Q_k) + \sum_{k=1}^{m} f(b_k)G(b_k).\]

From Theorem \ref{Thm:Additivity}, we have:
\[\int_K f\, dG = \sum_{k=1}^{m} \int_{a_k}^{b_k} f\, dG + \sum_{k=0}^{m} \int_{b_k}^{a_{k+1}} f\, dG - \sum_{k=1}^{m} f(a_k)G(a_k) - \sum_{k=1}^{m} f(b_k)G(b_k).\]

Combining all the information above, we obtain:
\begin{align*}
&\left| \sum_{k=1}^m \left( f(t_k)(G(b_k) - G(a_k)) + f(a_k)G(a_k) - \int_{a_k}^{b_k} f\, dG \right) \right| \\
&= \left| S(f, G, P^{\prime}) - \int_K f\, dG - \sum_{k=0}^{m} S(f|_{[b_k, a_{k+1}]}, G, Q_k) + \sum_{k=0}^{m} \int_{b_k}^{a_{k+1}} f\, dG \right| \\
&< \left| S(f, G, P^{\prime}) - \int_K f\, dG \right| + \left| \sum_{k=0}^{m} \left( S(f|_{[b_k, a_{k+1}]}, G, Q_k) - \int_{b_k}^{a_{k+1}} f\, dG \right) \right| \\
&< 2\varepsilon.
\end{align*}
\end{proof}

\begin{corollary}\label{corol:SH}
    Under the assumptions of Lemma \ref{lem:SH},
    \begin{equation}\label{eq:absolute SH_condition}
 \sum_{k=1}^m\left|  f(t_k)(G(b_k) - G(a_k)) + f(a_k)G(a_k)-\int_{a_k}^{b_k} f\, dG  \right| \leq 4\varepsilon.
\end{equation}
\end{corollary}
\begin{proof}Once again, let $\varepsilon > 0$ be arbitrary. Pick an $\varepsilon$-gauge $\delta$ for the function $f$, and let $\mathcal{S} = \{([a_k, b_k], t_k) : k = 1, \ldots, m\}$ be any $\delta$-fine tagged system of $K$. Define $\mathcal{S}^+ = \{([a_k, b_k], t_k) : k \in J^+\}$ as the subset of $\mathcal{S}$ consisting of all components $([a_k, b_k], t_k)$ for which
\[f(t_k)(G(b_k) - G(a_k)) + f(a_k)G(a_k) - \int_{a_k}^{b_k} f \, dG \geq 0,\]
and define $\mathcal{S}^- = \{([a_k, b_k], t_k) : k \in J^-\}$ as the subset consisting of all components for which
\[f(t_k)(G(b_k) - G(a_k)) + f(a_k)G(a_k) - \int_{a_k}^{b_k} f \, dG < 0.\]

Applying Lemma \ref{lem:SH} for both $\mathcal{S}^+$ and $\mathcal{S}^-$ we obtain 
\begin{align*}
 \sum_{k\in J^+}^m&\left|f(t_k)(G(b_k) - G(a_k)) + f(a_k)G(a_k)-\int_{a_k}^{b_k} f\, dG  \right|\\
 &=\left|\sum_{k\in J^+}^mf(t_k)(G(b_k) - G(a_k)) + f(a_k)G(a_k)-\int_{a_k}^{b_k} f\, dG  \right|<2\varepsilon
\end{align*}
and 
\begin{align*}
 \sum_{k\in J^-}^m&\left|f(t_k)(G(b_k) - G(a_k)) + f(a_k)G(a_k)-\int_{a_k}^{b_k} f\, dG  \right|\\
  &=-\left(\sum_{k\in J^-}^mf(t_k)(G(b_k) - G(a_k)) + f(a_k)G(a_k)-\int_{a_k}^{b_k} f\, dG  \right)\\
 &=\left|\sum_{k\in J^-}^mf(t_k)(G(b_k) - G(a_k)) + f(a_k)G(a_k)-\int_{a_k}^{b_k} f\, dG  \right|<2\varepsilon.
\end{align*}
Combining the above relations, we conclude the proof.
\end{proof}

\subsection{Connection with Radon Measures}\label{subsec:Radon} 
From Theorem~\ref{Thm:ContinuousIsGintegrable}, we know that if $G$ is a function of bounded variation and $f: K \to \mathbb{R}$ is continuous, then $f$ is $G$-integrable and satisfies the estimate
\[\left|\int_K f \, dG\right| \leq \sup_{x \in K} |f(x)|\left(|G(0_K)| + \operatorname{Var}(G)\right).\]
This inequality shows that the map $f \mapsto \int_K f \, dG$ defines a bounded linear functional on the Banach space $C(K)$ of continuous functions on $K$, and hence belongs to its dual space $C(K)^*$. By the Riesz--Kakutani Representation Theorem, there exists a unique Radon measure $\mu \in \mathcal{M}(K)$ such that
\begin{equation}\label{Rel-Radon}
\int_K f \, dG = \int_K f \, d\mu
\end{equation}
for every $f \in C(K)$. Thus, each function $G$ of bounded variation naturally induces a unique Radon measure $\mu \in \mathcal{M}(K)$ via this integral representation. Moreover, as shown in \cite[Theorem 2.4.11]{ronchim2021study}, among all functions of bounded variation that give rise to a given measure $\mu$, there exists a unique representative $G \in \operatorname{NBV}(K)$ that induces $\mu$.

More generally, given any real-valued function $G$ defined on a compact line $K$, we may associate to it a set function $\mu_G$ defined on the family of initial segments $\{[0_K, x] : x \in K\}$ by $\mu_G([0_K,x]) = G(x)$.
This rule uniquely determines a finitely additive set function on the algebra generated by the intervals of the form $[0_K, x]$, satisfying the identities
\begin{equation}\label{Rel-Radon2}\mu_G(\{0_K\}) = G(0_K), \quad \mu_G((y,x]) = G(x) - G(y) \quad \text{for } y < x. 
\end{equation}
In particular, for any function $f: K \to \mathbb{R}$ and any partition $P = \{([x_{i-1}, x_i], t_i) : 1 \leq i \leq n\}$ of $K$, the Riemann sum \eqref{eq:riemann} can be rewritten as
\[S(f,G,P) = f(0_K) \mu_G(\{0_K\}) + \sum_{i=1}^n f(t_i) \mu_G((x_{i-1}, x_i]).\]

However, it is not difficult to see that for a general function $G$, the set function $\mu_G$ cannot always be extended to a finite $\sigma$-additive Borel measure on $K$. Even assuming that $G$ is amenable is not sufficient for this purpose, as the following example will illustrate.

\begin{example}\label{ex:[0,omega]}
Let $\sum_{i=0}^\infty a_i$ be a series with nonzero real terms converging to $a_\omega\in \mathbb{R}$. Define $K=[0,\omega]$ and $G:[0,\omega]\to \mathbb{R}$ by
\[
G(n)=\sum_{i=1}^n a_i \quad \text{for } n\in \mathbb{N}, \quad \text{and} \quad G(\omega)=a_\omega.
\]
Note that $G$ is continuous and, in particular, amenable. However, if the series is only conditionally convergent, the measure $\mu_G$ cannot be extended to a signed, finite, $\sigma$-additive measure on the Borel sets. This is due to the fact that $\sum_{a_i>0} a_i=\infty$, which implies that the Borel set $\{i: a_i>0\}$ cannot be $\mu_G$-measurable. On the other hand, if the series is unconditionally convergent, then $G$ belongs to $\mathrm{NBV}(K)$, and it follows that $\mu_G$ is the Radon measure in $\mathcal{M}(K)$ defined by
\[\mu_G(S)=\sum_{i\in S} a_i,\quad \text{for each } S\subset [0,\omega].\]

In either case, consider a function $f:[0,\omega]\to \mathbb{R}$ such that the series $\sum_{i=0}^\infty f(i)a_i$ is conditionally convergent, and $f(\omega)=0$. We claim that $f$ is $G$-integrable and that
\[\int_0^\omega f\,dG=\sum_{i=0}^\infty f(i)a_i.\]
Indeed, write $b=\sum_{i=0}^\infty f(i)a_i$. Let $\varepsilon>0$ and choose $N\in\mathbb{N}$ such that for all $n>N$, holds $|b-\sum_{i=0}^{n} f(i)a_i|<\varepsilon$.
Define a gauge on $[0,\omega]$ by setting $\delta(n)=\{n\}$ for each $n\in[0,\omega)$ and $\delta(\omega)=(N,\omega]$. Let $P$ be an arbitrary $\delta$-fine partition. Note that $P$ must be of the form $\{([0,1],1),\dots,([n-1,n],n),\, ([n,\omega],\omega)\}$, with $n>N$. Computing the Riemann sum we have
\[S(f,G,P)=\sum_{i=0}^{n} f(i)a_i.\]
Thus,
\[\bigl|S(f,G,P)- b\bigr|= \Bigl|b-\sum_{i=0}^{n} f(i)a_i\Bigr|<\varepsilon,\]
which establishes our claim.

We now observe that when $G$ has bounded variation (that is, when the series $\sum a_i$ converges unconditionally), the function $f$ cannot be integrable with respect to the measure $\mu_G$. Otherwise, integral of $f$ over the Borel set $\{i<\omega : f(i)a_i>0\}$ would be infinite, a contradiction. This shows that even though the function $G$ induces the Radon measure $\mu_G$, there exist functions that are $G$-integrable yet not integrable with respect to $\mu_G$ in the usual sense of measure theory. In a subsequent subsection (see Theorem~\ref{Thm:CharacterizationGRadon}), we will show that every Borel-measurable function which is $\mu_G$-integrable is also $G$-integrable, and that in fact
\[\int_K f\,dG \;=\; \int_K f\,d\mu_G.\]
\end{example}

From now on, whenever $G \in \mathrm{NBV}(K)$, we denote by $\mu_G$ the unique Radon measure associated with $G$ such that equation~\eqref{Rel-Radon} holds for every continuous function $f: K \to \mathbb{R}$. In addition to the relations in~\eqref{Rel-Radon2} and the results from Section~\ref{Sec-BasicProperties}, we also have
\[\mu_G([0_K, x)) = L_G(x) \quad \text{and} \quad \mu_G([y, x]) = G(x) - L_G(y), \quad \text{for all } y < x.\]

If $G \in \mathrm{NBV}(K)$ and $\mu_G \in \mathcal{M}(K)$ is the Radon measure associated with $G$, we denote by $|\mu_G|$ the total variation measure of $\mu_G$. Note that when $G$ is positive and nondecreasing, we have $\mu_G = |\mu_G|$. For every subset $S \subset K$, we define the \emph{outer measure} of $S$ as
\[|\mu_G|^*(S) = \inf\left\{\sum_{n=1}^\infty |\mu_G|(I_n) : I_n \in \mathcal{O},\ S \subset \bigcup_n I_n \right\},\]
where $\mathcal{O}$ denotes the collection of open intervals in $K$.

This leads us to the following definition:

\begin{defin}
Let $K$ be a compact line, and let $G: K \to \mathbb{R}$ be an element of $\mathrm{NBV}(K)$. A subset $S \subset K$ is said to be \emph{$G$-null} if $|\mu_G|^*(S) = 0$. A function $f: K \to \mathbb{R}$ is called \emph{$G$-null} if the set $\{x \in K : f(x) \neq 0\}$ is $G$-null.
\end{defin}

\begin{prop}\label{Prop:Gnull}
Let $K$ be a compact line and let $G$ be a function in $\mathrm{NBV}(K)$. If $f: K \to \mathbb{R}$ is a $G$-null function, then $f$ is $G$-integrable and $\int_K f\,dG = 0$. 
\end{prop}

\begin{proof}
Let $\mu_G$ be the Radon measure associated with the function $G$. Since $\mu_G(\{0_K\}) = G(0_K)$, we have either $f(0_K) = 0$ or $G(0_K) = 0$. Let $Z = \{x \in K : f(x) \neq 0\}$, and define $Z_n = \{x \in Z : n - 1 \leq |f(x)| < n\}$. Note that $Z$ is the disjoint union of the sets $Z_n$. For each $n \in \mathbb{N}$, there exists a sequence of open intervals $(I_{nk})_k$ such that $Z_n \subset \bigcup_k I_{nk}$ and $\sum_k |\mu_G|(I_{nk}) < \frac{\varepsilon}{n 2^n}$. 

For each $x \in Z$, denote by $n(x)$ the unique natural number $n$ such that $x \in Z_n$, and let $k(x)$ be the smallest index $k$ such that $x \in I_{n(x)k}$. Define a gauge $\delta$ on $K$ by setting
\[\delta(x) = 
\begin{cases}
    K, & \text{if } x \in K \setminus Z, \\
    I_{n(x)k(x)}, & \text{if } x \in Z.
\end{cases}\]

Let $P = \{([x_{i-1}, x_i], t_i) : 1 \leq i \leq n\}$ be an arbitrary $\delta$-fine partition of $K$. Observe that whenever $t_i \in K \setminus Z$, the corresponding Riemann summand is zero. Fix $m \in \mathbb{N}$. Then for each $k \in \mathbb{N}$, whenever $t_i \in Z_m \cap I_{mk}$, we have $(x_{i-1}, x_i] \subset I_{mk}$, so
\[\left| \sum_{t_i \in Z_m \cap I_{mk}} G(x_i) - G(x_{i-1}) \right| 
\leq \sum_{t_i \in Z_m \cap I_{mk}} |\mu_G|(J_i) 
\leq |\mu_G|(I_{mk}).\]

Since $Z_m \subset \bigcup_k I_{mk}$, and $|f(t_i)| < m$ whenever $t_i \in Z_m$, we have
\[\left| \sum_{t_i \in Z_m} f(t_i)(G(x_i) - G(x_{i-1})) \right| 
\leq m \sum_{t_i \in Z_m} |\mu_G|(J_i) 
\leq m \cdot \frac{\varepsilon}{m 2^m} = \frac{\varepsilon}{2^m}.\]

It follows that the total contribution of the Riemann summands with $t_i \in Z$ is at most $\sum_m \frac{\varepsilon}{2^m} = \varepsilon$. We conclude that $|S(f,G,P)| \leq \varepsilon$, which implies that $f$ is $G$-integrable and that $\int_K f\,dG = 0$.
\end{proof}

Let $f_1, f_2, G: K \to \mathbb{R}$, where $K$ is a compact line and $G \in \mathrm{NBV}(K)$. We say that $f_1$ and $f_2$ are equal \emph{$G$-almost everywhere} if the set $\{ x \in K : f_1(x) \neq f_2(x) \}$ is $G$-null. More generally, we say that a certain property related to a subset of $K$ holds $G$-almost everywhere if the set of points where the property fails is $G$-null.

The example above, combined with the linearity of the integral (Proposition~\ref{Prop:BilinearityOfHKIntegral}), yields the following result.

\begin{prop}\label{prop:Gae}
    Let $K$ be a compact line, let $G \in \mathrm{NBV}(K)$, and let $f_1, f_2: K \to \mathbb{R}$ be functions that are equal $G$-almost everywhere. Then $f_1$ is $G$-integrable if and only if $f_2$ is $G$-integrable, in which case the values of their integrals with respect to $G$ are equal.
\end{prop}

In particular, this means that a function only needs to be defined $G$-almost everywhere in order for its $G$-integrability to be studied.

In what follows, we develop a Vitali-type covering theorem that will aid in proving that certain sets are $G$-null and will be instrumental in establishing the version of the Fundamental Theorem of Calculus given in Theorem~\ref{Thm:differentiation}.

\begin{lemma}\label{lem:vitali}
Let $K$ be a compact line equipped with a positive Radon measure $\mu$, and let $\mathcal F$ be a collection of intervals in $K$. Then there exists a subcollection $\mathcal G$ of pairwise disjoint elements of $\mathcal F$ with the following properties: 
    \begin{enumerate}[label=\textnormal{(F\alph*)}]
        \item \label{it:Admissible1} For every $I \in \mathcal{F}$, there exists $J \in \mathcal{G}$ such that $I \cap J \neq \emptyset$ and $2\mu(J) \geq \mu(I)$.
        \item \label{it:Admissible2} There exists a mapping $\varphi: \mathcal{G} \to \{\text{intervals in } K\}$ such that, for every $J \in \mathcal{G}$, $\mu(\varphi(J)) \leq 5\mu(J)$, and for every $I \in \mathcal{F}$, if $I \cap J \neq \emptyset$ and $2\mu(J) \geq \mu(I)$, then $I \subset \varphi(J)$.
    \end{enumerate}
\end{lemma}
\begin{proof}
The collection $\mathcal F$ can be decomposed as the disjoint union of $\mathcal F_n$, $n\in\mathbb N\cup\{\infty\}$, where
\[\mathcal F_n = \left\{I\in\mathcal F: \frac{\mu(K)}{2^n} < \mu(I) \leq \frac{\mu(K)}{2^{n-1}}\right\}, \quad n\in\mathbb N,\]
and
\[\mathcal F_\infty = \{I\in\mathcal F: \mu(I)=0\}.\]
We define $\mathcal G_n \subset \mathcal F_n$ recursively. First, let $\mathcal G_1 \subset \mathcal F_1$ be a maximal pairwise disjoint subcollection, which exists by Zorn's Lemma. Given $k \geq 2$, suppose that $\mathcal G_j$ has been defined for all $j \in \{1,\dots,k-1\}$, and consider
\[\mathcal H_k = \left\{I \in \mathcal F_k : \forall J \in \mathcal G_1 \cup \dots \cup \mathcal G_{k-1},\, I \cap J = \emptyset \right\}.\]
Apply Zorn's Lemma again to obtain a maximal pairwise disjoint subcollection $\mathcal G_k \subset \mathcal H_k$.

Now consider
\[\mathcal H_\infty = \left\{I \in \mathcal F_\infty : \forall J \in \bigcup_{n\in\mathbb N} \mathcal G_n,\, I \cap J = \emptyset \right\},\]
and once more use Zorn's Lemma to obtain a maximal pairwise disjoint subcollection 
$\mathcal G_\infty \subset \mathcal H_\infty$.

Define $\mathcal G := \bigcup_{n \in \mathbb N \cup \{\infty\}} \mathcal G_n$. We now verify that $\mathcal G$ satisfies conditions \ref{it:Admissible1} and \ref{it:Admissible2}.

Fix $n \in \mathbb N$ and $I \in \mathcal F_n$. Then one of the following holds:
\begin{enumerate}
    \item[(i)] $I \in \mathcal G_n$,
    \item[(ii)] $I \in \mathcal H_n \setminus \mathcal G_n$, hence $I$ intersects some $J \in \mathcal G_n$,
    \item[(iii)] $I \notin \mathcal H_n$, hence $I$ intersects some $J \in \mathcal G_1 \cup \dots \cup \mathcal G_{n-1}$.
\end{enumerate}
In all cases, $I$ intersects some $J \in \mathcal G_1 \cup \dots \cup \mathcal G_n$. On one hand,
\[\mu(J) > \frac{\mu(K)}{2^n};\]
on the other hand, $\mu(I) \leq \mu(K)/2^{n-1}$. From these, we deduce that
\begin{equation}\label{eq:Vitali2}
\mu(J) > \frac{\mu(I)}{2}.
\end{equation}

A similar reasoning applies for $I \in \mathcal F_\infty$, where we replace the inequality in \eqref{eq:Vitali2} with $\geq$ (since $\mu(J)$ may be zero). This shows that $\mathcal G$ satisfies \ref{it:Admissible1}.

To verify \ref{it:Admissible2}, define
\[\varphi(J) := J \cup \bigcup \left\{ I \in \mathcal{F} : I \cap J \neq \emptyset \text{ and } 2\mu(J) \geq \mu(I) \right\}.\]
It is clear from the definition that whenever $I \in \mathcal{F}$ intersects $J \in \mathcal{G}$ and $2\mu(J) \geq \mu(I)$, we have $I \subset \varphi(J)$.

Now fix $J \in \mathcal{G}$, and let $[x,y] \subset \varphi(J)$. Suppose $x \in I_1$ and $y \in I_2$, with $I_1, I_2 \in \mathcal{F}$ intersecting $J$ and satisfying $2\mu(J) \geq \mu(I_1), \mu(I_2)$. Then $[x,y] \subset I_1 \cup J \cup I_2$,
and so
\begin{align*}
    \mu([x,y]) &\leq \mu([x,y] \cap I_1) + \mu([x,y] \cap J) + \mu([x,y] \cap I_2) \\
    &\leq \mu(I_1) + \mu(J) + \mu(I_2) \\
    &\leq 2\mu(J) + \mu(J) + 2\mu(J) = 5\mu(J).
\end{align*}
Since $\mu$ is inner regular, we conclude that $\mu(\varphi(J)) \leq 5\mu(J)$, completing the proof.
\end{proof}

\begin{defin}\label{Def:AdmissibleCovering}
Let $K$ be a compact line, and let $A$ be a subset of $K$. An \emph{admissible covering} of $A$ is a family $\mathcal{F}$ of intervals of $K$ such that for every $a \in A$ and every finite subcollection $\{J_1,\ldots,J_k\}$ of pairwise disjoint elements of $\mathcal{F}$, if $a\in A\setminus \bigcup_{i=1}^k J_i$, then there exists an interval $I \in \mathcal{F}$ such that $a \in I$ and $I\cap \big(\bigcup_{i=1}^k J_i\big)=\emptyset$.
\end{defin}

\begin{theorem}\label{Thm:vitali}
Let $K$ be a compact line equipped with a positive Radon measure $\mu$, and denote by $\mu^*$ the outer measure induced by $\mu$. Let $A$ be a subset of $K$, and let $\mathcal{F}$ be an admissible covering of $A$. Then, for every $\varepsilon > 0$, there exists a finite subcollection $\mathcal{H} \subset \mathcal{F}$ consisting of pairwise disjoint elements such that
\begin{equation}\label{eq:thesis Vitali}
\mu^*(A \setminus \bigcup \mathcal{H}) < \varepsilon.
\end{equation}
\end{theorem}

\begin{proof}
Let $\varepsilon > 0$, and consider the pairwise disjoint subcollection $\mathcal{G} \subset \mathcal{F}$ obtained as in Lemma~\ref{lem:vitali}. Since $\mathcal{G}$ is pairwise disjoint, we have $\sum_{J \in \mathcal{G}} \mu(J) \leq \mu(K) < \infty$,
where $\sum_{J \in \mathcal{G}} \mu(J) = \sup\left\{\sum_{J \in \mathcal{G}'} \mu(J) : \mathcal{G}' \subset \mathcal{G},\ \mathcal{G}' \text{ finite} \right\}$. Choose $\delta > 0$ such that
\[\sum_{J \in \mathcal{G}_{<\delta}} \mu(J) < \frac{\varepsilon}{5},\]
where $\mathcal{G}_{<\delta} = \{J \in \mathcal{G} : \mu(J) < \delta\}$. Let $\mathcal{G}_{\geq \delta} = \{J \in \mathcal{G} : \mu(J) \geq \delta\}$ and note that this set is finite. If $A \setminus \bigcup \mathcal{G} = \emptyset$, i.e., if $\mathcal{G}$ covers $A$, then $A \setminus \bigcup \mathcal{G}_{\geq \delta} \subset \bigcup \mathcal{G}_{<\delta},$
so
\[\mu^*(A \setminus \bigcup \mathcal{G}_{\geq \delta}) \leq \mu^*(\bigcup \mathcal{G}_{<\delta}) \leq \sum_{J \in \mathcal{G}_{<\delta}} \mu(J) < \frac{\varepsilon}{5}.\]
Hence, inequality~\eqref{eq:thesis Vitali} holds with $\mathcal{H} = \mathcal{G}_{\geq \delta}$.

Now suppose that $A \setminus \bigcup \mathcal{G} \neq \emptyset$. Since $\mathcal{G}_{\geq \delta}$ is finite, by the definition of admissible covering, for each $a \in A \setminus \bigcup \mathcal{G}_{\geq \delta}$ there exists an interval $I_a \in \mathcal{F}$ such that $a \in I_a$ and $I_a\cap\big(\bigcup \mathcal{G}_{\geq \delta}\big)=\emptyset$. By property~\ref{it:Admissible1}, we can find $J_a \in \mathcal{G}$ such that $I_a \cap J_a \neq \emptyset$ and $2\mu(J_a) \geq \mu(I_a)$. Moreover, we have $J_a \in \mathcal{G}_{<\delta}$. Define
\[
\mathcal{J} = \{J_a : a \in A \setminus \bigcup \mathcal{G}_{\geq \delta}\},
\]
and note that $\mathcal{J} \subset \mathcal{G}_{<\delta}$. By property~\ref{it:Admissible2}, we obtain
\[
A \setminus \bigcup \mathcal{G}_{\geq \delta} \subset \bigcup_{a \in A \setminus \bigcup \mathcal{G}_{\geq \delta}} I_a \subset \bigcup_{a \in A \setminus \bigcup \mathcal{G}_{\geq \delta}} \varphi(J_a) = \bigcup_{J \in \mathcal{J}} \varphi(J),
\]
and therefore
\[
\mu^*(A \setminus \bigcup \mathcal{G}_{\geq \delta}) \leq \mu^*\left(\bigcup_{J \in \mathcal{J}} \varphi(J)\right) \leq \sum_{J \in \mathcal{J}} \mu(\varphi(J)) \leq \sum_{J \in \mathcal{J}} 5\mu(J) < \varepsilon.
\]
Thus, setting again $\mathcal{H} = \mathcal{G}_{\geq \delta}$ concludes the proof.
\end{proof}

\section{Calculus on compact lines}\label{sec:Calculus}

\subsection{Absolutely integrable functions}In this subsection, our goal is to establish important properties of primitive functions under the assumption that the integrator $G$ is at least \emph{amenable}. This is the content of Theorem~\ref{Thm:Amenability} below. 

In addition, we provide a characterization of absolutely integrable functions. We say that a function $f$ is \emph{absolutely $G$-integrable} if $|f|$ is also $G$-integrable. As in the case of the classical Kurzweil--Stieltjes integral theory, it is not true in general that every $G$-integrable function is absolutely $G$-integrable, see Example \ref{ex:[0,omega]}. In Theorem \ref{Thm:CaracterizacaodasFuncoesAbsolutamenteIntegraveis} we adapt the well-known \cite[Theorem~7.5]{bartle2001modern} to our integration setting.

\begin{theorem}[Amenability of the indefinite integral]\label{Thm:Amenability}
    Let $K$ be a compact line and let $G: K \to \mathbb{R}$ be an amenable function. If $f:K\to \mathbb{R}$ is $G$-integrable, then the function $F:K\to \mathbb{R}$ defined by the formula
    \[F(x) = \int_{0_K}^x f\, dG\]
    is a well defined amenable function. Moreover, for every $c\in K$ holds
    \[F(c)-L_F(c)=f(c)\big(G(c)-L_G(c)\big).\] 
In particular, if $G$ is continuous at $c$, then $F$ is continuous at $c$.   
\end{theorem}
\begin{proof}
If $f$ is $G$-integrable, then according to Propositions \ref{Prop:IndicatorFunctionIntegral} and \ref{prop:IntegrationOnASubinterval}, the function $F$, stated as above, is well defined. To check that $F$ is right-continuous, let $ c \in K $ be an arbitrary point, and assume that $ c $ is right-dense. Let $\delta$ be an $\frac{\varepsilon}{4}$-gauge for $ f $ that exposes the point $ c $. Since $ G $ is right-continuous, we may assume that $\delta(c) \cap [c, 1_K] = [c, v)$ is such that $ |G(x) - G(c)| < \frac{\varepsilon}{2(|f(c)| + 1)} $ for every $ x \in [c, v) $.

For any $ x \in [c, v) $, by applying Lemma~\ref{lem:SH} (Saks--Henstock) to the tagged system $\mathcal{S} = \{([c, x], c)\}$, and the functions $ f $ and $ G $, we obtain:
\[\left| \int_c^x f \, dG - f(c) G(c) \right| - \left| f(c)(G(x) - G(c)) \right| \leq \left| f(c)(G(x) - G(c)) + f(c) G(c) - \int_c^x f \, dG \right| < \frac{\varepsilon}{2}.\]

Hence, by applying Theorem \ref{Thm:Additivity},
\begin{align*}
|F(x) - F(c)| &= \left| \int_0^x f \, dG - \int_0^c f \, dG \right|= \left| \int_c^x f \, dG - f(c) G(c) \right|\\
&< \left| f(c)(G(x) - G(c)) \right| + \frac{\varepsilon}{2}< \frac{\varepsilon}{2} + \frac{\varepsilon}{2} = \varepsilon.
\end{align*}

For the second part, let $c \in K$ be arbitrary. If $c$ is left-isolated and different from $0_K$, then by Theorem \ref{Thm:Additivity}, we obtain  
\begin{align*}
F(c) - L_F(c) &= F(c) - F(c^-) = \int_{c^-}^{c} f \, dG - f(c^-) G(c^-) \\
&= f(c) \big( G(c) - G(c^-) \big) = f(c) \big( G(c) - L_G(c) \big).
\end{align*}
In the case $c = 0_K$, a straightforward adaptation of the previous argument yields the desired result.

If $c$ is a left-dense point, let $\varepsilon >0$ be arbitrary. We let $\delta$ be a $\frac{\epsilon}{2}$-gauge for $f$ on $K$ exposing the point $c$. By denoting $\operatorname{L}_G(c)=\ell$, we may also assume that $\delta(c)\cap [0_K,c)= (u,c)$ and $|\ell-G(x)|<\frac{\varepsilon}{2(|f(c)|+1)}$ for every $x \in (u,c)$. Given $x \in (u,c)$. By applying Lemma~\ref{lem:SH} (Saks--Henstock) for the tagged system $\mathcal{S}=\{([x,c],c)\}$ we obtain 
\[\left| f(c)(G(c) - G(x)) + f(x) G(x) - \int_x^c f \, dG \right| < \frac{\varepsilon}{2}.\]
From Theorem \ref{Thm:Additivity} we have
\small\begin{align*}
|F(x) - (F(c) + f(c)\int_K\chi_{\{c\}}\,dG)|&= \left| f(x)G(x) - \int_x^c f \, dG + f(c)(G(c) - G(x))+ f(c)(G(x) - \ell) \right| \\
& \leq \left|f(c)(G(c) - G(x)) +f(x)G(x) - \int_x^c f \, dG\right|+|f(c)(G(x) - \ell)|\\
&<\frac{\varepsilon}{2}+\frac{\varepsilon}{2} <\varepsilon.
\end{align*}\normalsize
We deduce that 
\[L_F(c)=\lim_{y \nearrow c}F(y)=F(c)+f(c)\int_K\chi_{\{c\}}\,dG=F(c)+f(c)(G(c)-L_G(c)).\] 
This completes the proof.
\end{proof}

\begin{theorem}\label{Thm:CaracterizacaodasFuncoesAbsolutamenteIntegraveis}
Let $K$ be a compact line, and let $G: K \to \mathbb{R}$ be a nondecreasing amenable function. Suppose $f$ is a $G$-integrable function, and consider the fuction $F: K \to \mathbb{R}$ defined by the formula 
\[F(x) = \int_{0_K}^x f \, dG.\] 
Then, $|f|$ is $G$-integrable if and only if $F\in \mathrm{NBV}(K)$. In this case, we have
\[\int_K |f| \, dG = |f(0_K)|G(0_K) + \operatorname{Var}(F).\]
\end{theorem}
\begin{proof}
Let us first assume that $|f|$ is $G$-integrable. According to Theorem~\ref{Thm:Amenability}, the primitive function $F$ is amenable. Therefore, it remains to show that $F$ has bounded variation. Let $\{x_0, x_1, \ldots, x_n\}$ be an arbitrary division of $K$. Define $I_0 = \{0_K\}$ and, for each $1 \leq i \leq n$, let $I_i = (x_{i-1}, x_i]$. From the results in Subsection~\ref{Sec-BasicProperties}, we obtain:
\begin{align*}
F(x_i) - F(x_{i-1}) 
&= \int_{x_{i-1}}^{x_i} f \, dG - f(x_{i-1}) G(x_{i-1}) \\
&= \left( \int_{x_{i-1}}^{x_i} f \, dG - f(x_{i-1}) L_G(x_{i-1}) \right) 
- f(x_{i-1}) \left( G(x_{i-1}) - L_G(x_{i-1}) \right) \\
&= \int_K f|_{I_i} \, dG.
\end{align*}
Therefore
\begin{align*}
|f(0_K)|G(0_K) &+\sum_{i=1}^n \left| F(x_i) - F(x_{i-1}) \right|= |f(0_K)|G(0_K)+\sum_{i=1}^n | \int_K f|_{I_i} \, dG |\\
&\leq \sum_{i=0}^n \int_{K} |f||_{I_i} \, dG = \int_K |f| \, dG
\end{align*}
and we deduce that $\operatorname{Var}(F) \leq \int_K |f| \, dG -|f(0_K)|G(0_K)< \infty$.

Conversely, assume that $\operatorname{Var}(F) < \infty$. Given $\varepsilon > 0$, let $\{x_0, x_1, \ldots, x_n\}$ be a division of $K$ such that $x_0<x_1<\ldots<x_n$ and
\[\operatorname{Var}(F) - \varepsilon 
< \sum_{i=1}^n \left| F(x_i) - F(x_{i-1}) \right| 
\leq \operatorname{Var}(F).\]

Let $\delta$ be an $\varepsilon$-gauge for $f$, and let $P$ be a $\delta$-fine partition that exposes the points $\{x_1, \ldots, x_n\}$. By refining the partition $P$ we may assume that all the points $\{x_0, x_1, \ldots, x_n\}$ are division points of the partition $P$ (see Remark~\ref{Rem:EnlargePartitions}).

Assuming $P = \{([y_{j-1},y_j]):1\leq j \leq m\}$, the following inequality holds:
\begin{align*}
\operatorname{Var}(F) - \varepsilon &< \sum_{i=1}^n \left| F(x_i) - F(x_{i-1}) \right|\leq\sum_{j=1}^m \left| F(y_j) - F(y_{j-1}) \right|\\
&= \sum_{j=1}^m \left| \int_{y_{j-1}}^{y_j} f \, dG - f(y_{j-1}) G(y_{j-1}) \right| 
\leq \operatorname{Var}(F).
\end{align*}

By applying Corollary \ref{corol:SH} for the partition $P$, using the reverse triangle inequality, and noting that $G$ is nondecreasing, we obtain:
\[\left|\sum_{j=1}^m |f(t_j)| (G(y_j) - G(y_{j-1})) - \sum_{j=1}^m | \int_{y_{j-1}}^{y_j} f \, dG - f(y_{j-1}) G(y_{j-1}) | \right|\leq 4\varepsilon.\]

Thus, we deduce:
\begin{align*}
|S(|f|, G, P) &- \big(|f(0_K)| G(0_K) + \operatorname{Var}(F)\big)| 
\leq \left| \sum_{j=1}^m |f(t_j)| (G(y_j) - G(y_{j-1})) \right. \\
&- \left. \sum_{j=1}^m | \int_{y_{j-1}}^{y_j} f \, dG - f(y_{j-1}) G(y_{j-1}) | \right|+ \left| \sum_{j=1}^m | F(y_j) - F(y_{j-1}) | - \operatorname{Var}(F) \right| \\
&< 4\varepsilon + \varepsilon = 5\varepsilon.
\end{align*}
and this completes the proof.
\end{proof}

\begin{corollary}\label{Cor:|f|IsIntegrable}
Let $K$ be a compact line, and let $G: K \to \mathbb{R}$ be a nondecreasing amenable function. Suppose $f, g: K \to \mathbb{R}$ are $G$-integrable functions, and assume that $|f(x)| \leq g(x)$ for every $x \in K$. Then $|f|$ is $G$-integrable.
\end{corollary}
\begin{proof}
Define $F: K \to \mathbb{R}$ by $F(x) = \int_{0_K}^x f \, dG$, and let $\{x_0, x_1, \ldots, x_n\}$ be a division of $K$. Set $I_0 = \{0_K\}$ and, for each $1 \leq i \leq n$, define $I_i = (x_{i-1}, x_i]$. Repeating the argument from the proof of Theorem~\ref{Thm:CaracterizacaodasFuncoesAbsolutamenteIntegraveis}, we may write:
\begin{align*}
\sum_{i=1}^n|F(x_i)-F(x_{i-1})|&=\sum_{i=1}^n \left| \int_K f|_{I_i} \, dG \right|\leq \sum_{i=1}^n \int_K |f||_{I_i} \, dG\\
&\leq \sum_{i=1}^n \int_K g|_{I_i} \, dG=\sum_{i=0}^n \int_K g|_{I_i} \, dG - g(0_K)G(0_K)\\
&=\int_K g\, dG- g(0_K)G(0_K)<\infty
\end{align*}
Thus, $\operatorname{Var}(F) < \infty$, which, by Theorem \ref{Thm:CaracterizacaodasFuncoesAbsolutamenteIntegraveis}, implies that $|f|$ is $G$-integrable.
\end{proof}

\subsection{$G$-differentiability on compact lines}
In this subsection, we introduce a notion of derivative that underpins a version of the Fundamental Theorem of Calculus. Our definition is inspired by \cite[Definition 1.1]{pouso2015new}; however, we assume the integrator $G$ to be nondecreasing and amenable, rather than non-decreasing and left-continuous at every point. For related concepts involving $\Delta$- and $\nabla$-derivatives on time scales, see \cite[Chapter 8]{bohner2001dynamic} and \cite[Chapter 8]{monteiro2019kurzweil}.

\begin{defin}\label{def:diff}  
Let $K$ be a compact line, and let $f, G: K \to \mathbb{R}$ be functions, with $G$ amenable and nondecreasing. Fix a point $x \in K$.

If $L_G(x) = G(x)$, we say that $f$ is \emph{$G$-differentiable at $x$} if (in case $x$ is left-isolated in $K$, we assume $f(x) = L_f(x)$ and) the following conditions are satisfied:
\begin{enumerate}[label=\textnormal{(D\alph*)}]
    \item \label{it:diff_defin1} For every neighborhood $V$ of $x$, the function $G$ is not constant on $V$.
    \item \label{it:diff_defin2} There exists $D \in \mathbb{R}$ such that, for every $\varepsilon > 0$, there exists a neighborhood $V$ of $x$ in $K$ such that, for all $y \in V$,
\[\left|f(y) - f(x) - D \big(G(y) - G(x)\big) \right| \leq \varepsilon \big|G(y) - G(x)\big|.\]
\end{enumerate}

If $L_G(x) \neq G(x)$, we say that $f$ is \emph{$G$-differentiable at $x$} if the following conditions are satisfied:
\begin{enumerate}[label=\textnormal{(E\alph*)}]
    \item \label{it:RElDiffAux1} $f$ is regulated at $x$.
    \item \label{it:RElDiffAux2} There exists $D \in \mathbb{R}$ such that, for every $\varepsilon > 0$, there exists a neighborhood $V$ of $x$ in $K$ such that, for all $y \in V\cap [x,1_K]$,
    \[
    \left|f(y) - L_f(x) - D \big(G(y) - L_G(x)\big) \right| \leq \varepsilon \big|G(y) - L_G(x)\big|.
    \]
\end{enumerate}
In both cases, we say that $D$ is the \emph{$G$-derivative} of $f$ at $x$, and we write $\frac{df}{dG}(x) = D$.    
\end{defin}

It is easy to verify that the $G$-derivative of a function, if it exists, is unique. Furthermore, we present the following characterization of $G$-differentiability in the case $L_G(x) \neq G(x)$.

\begin{prop}\label{Pro:AuXDiff}
Let $K$ be a compact line, and let $f, G: K \to \mathbb{R}$ be functions with $G$ nondecreasing and amenable. For a point $x \in K$ such that $L_G(x) \neq G(x)$, assume that $f$ is regulated at $x$. Then, the function $f$ is $G$-differentiable at $x$ if and only if $f$ is right-continuous at $x$. In this case,
\[\frac{df}{dG}(x)=\frac{f(x)-L_f(x)}{G(x)-L_G(x)}.\]
\end{prop}
\begin{proof}  
Since $G$ is nondecreasing and $L_G(x) \neq G(x)$, we can define a function $H: K \to \mathbb{R}$ for each $y \in K$ by the formula
\[
H(y) = 
\begin{cases}
\displaystyle\frac{f(y) - L_f(x)}{G(y) - L_G(x)} - \frac{f(x) - L_f(x)}{G(x) - L_G(x)} & \text{if } y \geq x, \\
0 & \text{otherwise}.
\end{cases}
\]

If $f$ is $G$-differentiable at $x$, then it is straightforward to check that $\frac{df}{dG}(x)=\frac{f(x) - L_f(x)}{G(x) - L_G(x)}$,
and consequently, the condition \ref{it:RElDiffAux2} implies that the function $H$ is right-continuous at $x$, with $H(x)=0$. Setting $P(y) = f(y) - L_f(x)$ and $Q(y) = G(y) - L_G(x)$, we can write, for every $y \geq x$,
\begin{align*}  
    |H(y)| &= \left| \frac{P(y)}{Q(y)} - \frac{P(x)}{Q(x)} \right| \geq \left| \frac{P(x) - P(y)}{Q(y)} \right| - |P(x)| \left| \frac{1}{Q(y)} - \frac{1}{Q(x)} \right|,
\end{align*}
which leads to
\begin{align*}  
    |f(x) - f(y)| &= |P(x) - P(y)| \leq |Q(y)||H(y)| + |D|\, |Q(x) - Q(y)| \\
    &= |G(y) - L_G(x)||H(y)| + |D|\, |G(x) - G(y)|.
\end{align*}
Since $G$ is amenable, it follows from the previous relation that $f$ is right-continuous at $x$.

Conversely, assume that $f$ is right-continuous at $x$. Since $G$ is also right-continuous at $x$ by hypothesis, it follows that the function $H$ is right-continuous at $x$, with $H(x)=0$. This implies that condition \ref{it:RElDiffAux2} holds with
\[D = \frac{f(x) - L_f(x)}{G(x) - L_G(x)}.\]
Therefore, $f$ is $G$-differentiable at $x$, with $\frac{df}{dG}(x) = D$. This concludes the proof.
\end{proof}

\begin{lemma}[Straddle]\label{lem:straddle}
Let $K$ be a compact line, and let $G: K \to \mathbb{R}$ be a non-decreasing, positive, and amenable function. Suppose $f:K \to \mathbb{R}$ is $G$-differentiable at a point $t \in K$. Then, in each of the following cases, for every $\varepsilon > 0$, there exists an open interval $I$ containing $t$ such that, for any $(x,y] \subset I$ with $x \leq t \leq y$, the following inequality holds:
\begin{enumerate}[label=\textnormal{(S\alph*)}]
    \item \label{it:Straddle1} If $G(t) = L_G(t)$, then
    \[\left| f(y) - f(x) - \frac{df}{dG}(t) \big(G(y) - G(x)\big) \right| \leq \varepsilon \big(G(y) - G(x)\big).\]
    \item \label{it:Straddle2} If $G(t) \neq L_G(t)$ and $t$ is left-isolated, then
    \[\left| f(y) - f(x) - \frac{df}{dG}(t) \big(G(y) - G(x)\big) \right| \leq \varepsilon \big(G(y) - L_G(x)\big).\]
\end{enumerate}
\end{lemma}
\begin{proof}
Assume that $G(t) = L_G(t)$ and let $\varepsilon > 0$ be arbitrary. Denoting $D = \frac{df}{dG}(t)$, there exists a neighborhood $V$ of $t$ such that condition~\ref{it:diff_defin2} holds for every $y \in V$. 

If $t$ is left-dense, we let $I = (a,b)$ (or $I = (a,t]$ in case $t = 1_K$) be such that $a < t < b$ and $[a,b) \subset V$. Then, for any $(x,y] \subset I$ with $x \leq t \leq y$, we have $x,y \in V$. By adding and subtracting $f(t) + D G(t)$ and performing standard computations, we obtain
\begin{align*}
\left| f(y) - f(x) - D\big(G(y) - G(x)\big) \right|
&\leq \varepsilon \big(G(y) - G(t)\big) + \varepsilon \big(G(t) - G(x)\big) \\
&= \varepsilon \big(G(y) - G(x)\big).
\end{align*}

If $t$ is left-isolated, then $t<1_K$ by condition \ref{it:diff_defin1}. We fix $I=[t,b) \subset V$ with $t<b$. Since in this case we have $G(t) = L_G(t)$ and $f(t) = L_f(t)$, for each $(x,y] \subset I$ with $x \leq t \leq y$, we have
\[\left| f(y) - f(x) - D \big(G(y) - G(x)\big) \right| \leq \varepsilon \big(G(y) - G(x)\big).\]

Now assume that $G(t) \neq L_G(t)$ and that $t$ is left-isolated. Again, letting $D = \frac{df}{dG}(t)$, we can choose an open interval $I \subset [t,1_K]$ containing $t$ such that condition~\ref{it:RElDiffAux2} holds for every $y \in I$, with error bound $\frac{\varepsilon}{2}$. For any $(x,y] \subset I$ with $x \leq t \leq y$, we consider two cases:

If $x < t$, then $x = t^-$, so
\begin{align*}
\left| f(y) - f(x) - D \big(G(y) - G(x)\big) \right|
&= \left| f(y) - L_f(t) - D\big(G(y) - L_G(t)\big) \right| \\
&\leq \varepsilon \big(G(y) - L_G(t)\big) \leq \varepsilon \big(G(y) - L_G(x)\big).
\end{align*}

If $x = t$, we add and subtract $L_f(t) + D L_G(t)$ to obtain
\begin{align*}
\left| f(y) - f(x) - D \big(G(y) - G(x)\big) \right|
&\leq \frac{\varepsilon}{2}\big(G(y) - L_G(t)\big) + \frac{\varepsilon}{2} \big(G(x) - L_G(t)\big) \\
&\leq \varepsilon \big(G(y) - L_G(t)\big) \leq \varepsilon \big(G(y) - L_G(x)\big).
\end{align*}

This concludes the proof.
\end{proof}

\subsection{The Fundamental Theorem of Calculus}

Motivated by \cite[Theorem 2.4]{peterson2006henstock}, we now establish the following formulation of the Fundamental Theorem of Calculus.

\begin{theorem}[Fundamental Theorem of Calculus: integrating derivatives]\label{Thm:integration}
Let $K$ be a compact line, and let $G: K \to \mathbb{R}$ be a non-decreasing, positive, and amenable function. Suppose that a continuous function $F: K \to \mathbb{R}$ is $G$-differentiable at every point of $K$, except possibly at $0_K$ and at a countable set of left-dense points. Let $f: K \to \mathbb{R}$ be any function such that $f(x) = \frac{dF}{dG}(x)$ at every point where $F$ is $G$-differentiable. Then, $f$ is $G$-integrable, and the following equality holds:
\[\int_K f \, dG = F(1_K) - \left(F(0_K) - f(0_K) G(0_K)\right).\]
\end{theorem}
\begin{proof}
Let $A$ be the set of points where $F$ is not $G$-differentiable. Since $F$ is continuous and each point of $A$ is either $0_K$ or left-dense, Proposition~\ref{Pro:AuXDiff} implies that $G$ is continuous at every $t \in A$. Let $B$ be the set of left-dense points $t \in K$ such that $L_G(t) \neq G(t)$. As $G$ is non-decreasing, $B$ is countable. Moreover, the continuity of $F$ ensures that $F$ is $G$-differentiable at each $t \in B$, and hence $f(t) = \frac{dF}{dG}(t) = 0$ by Proposition~\ref{Pro:AuXDiff}. Define $E = A \cup B = \{c_1, c_2, \dots\}$.

Let $\varepsilon > 0$ be arbitrary. We define a gauge $\delta$ on $J$ as follows: for each $t \in K \setminus E$, let $\delta(t)$ be an open interval obtained by applying the Straddle Lemma~\ref{lem:straddle} to $t$ and $\varepsilon$. If $t = c_j \in E$, then since $F$ is continuous, there exists an open neighborhood $I(c_j)$ of $c_j$ such that
\[
x \in I(c_j) \Rightarrow |F(x) - F(c_j)| < \frac{\varepsilon}{2^{j+3}}.
\]

Since each element of $E \setminus \{0_K\}$ is left-dense, we choose $\delta(c_j)$ to be an open interval of the form $(p, q)$ (or $(p, c_j]$ if $c_j = 1_K$) such that $p < c_j < q$ and $[p, q) \subset I(c_j)$. If $c_j = 0_K$, we take $\delta(c_j)$ to be a non-empty open interval of the form $[0_K, q) \subset I(c_j)$. Moreover, for those $j \in \mathbb{N}$ such that $G$ is continuous at $c_j$, we may also assume that
\[|G(x)-G(c_j)| < \frac{\varepsilon}{2^{j+3}(|f(c_j)|+1)} \quad \text{for every } x \in \delta(c_j).\]

Now let $P = \{([x_{i-1}, x_i], t_i) : 1 \leq i \leq n\}$ be a $\delta$-fine partition of $K$, and we may assume that $x_0 < x_1 < \dots < x_n$ (see Remark~\ref{Rem:EnlargePartitions}). 

If $t_i = c_j \in E$, then $(x_{i-1}, x_i] \subset \delta(t_i)$. By the construction of the gauge, we have $[x_{i-1}, x_i] \subset I(c_j)$, and consequently, in the case $c_j \in A$, we obtain:
\begin{align*}
    |F(x_i) - F(x_{i-1}) - &f(c_j) \big(G(x_i) - G(x_{i-1})\big)| \\
    &\leq |F(x_i) - F(c_j)| + |F(c_j) - F(x_{i-1})| \\
    &+ |f(c_j)|\big(|G(x_i) - G(c_j)| + |G(c_j) - G(x_{i-1})|\big) \\
    &\leq\frac{\varepsilon}{2^{j+3}} + \frac{\varepsilon}{2^{j+3}} + \frac{|f(c_j)|}{(|f(c_j)|+1)} \left( \frac{\varepsilon}{2^{j+3}} +\frac{\varepsilon}{2^{j+3}} \right) \leq \frac{\varepsilon}{2^{j+1}}.
\end{align*}

In the case $c_j \in B$, we obtain:
\begin{align*}
    |F(x_i) - F(x_{i-1}) - &f(c_j) \big(G(x_i) - G(x_{i-1})\big)|\\
    &\leq|F(x_i) - F(c_j)| + |F(c_j) - F(x_{i-1})| + |f(c_j)\big(G(x_i) - G(x_{i-1}\big)| \\
    &\leq\frac{\varepsilon}{2^{j+3}} + \frac{\varepsilon}{2^{j+3}} + 0 \leq \frac{\varepsilon}{2^{j+1}}.
\end{align*}

Since each point in $E$ can appear as a tag in at most two components of the partition $P$, it follows that
\[
\sum_{t_i \in E} |F(x_i) - F(x_{i-1}) - f(t_i)(G(x_i) - G(x_{i-1}))| < 2\sum_i \frac{\varepsilon}{2^{i+1}} = \varepsilon.
\]

Now let $t_{i_1} \leq t_{i_2} \leq \dots \leq t_{i_m}$ be all tag points in $P$ that are left-isolated and satisfy $G(t_{i_k}) \neq L_G(t_{i_k})$. For each $1 \leq k \leq m$, we have $(x_{i_k - 1}, x_{i_k}] \subset \delta(t_{i_k})$, with $x_{i_k - 1} \leq t_{i_k} \leq x_{i_k}$. Then, by the Straddle Lemma~\ref{it:Straddle2}, we obtain:
\[
\left|F(x_{i_k}) - F(x_{i_k - 1}) - f(t_{i_k}) (G(x_{i_k}) - G(x_{i_k - 1}))\right| \leq \varepsilon (G(x_{i_k}) - L_G(x_{i_k - 1})).
\]

Summing over $k$ gives:
\begin{align*}
\sum_{k=1}^m \big|F(x_{i_k}) &- F(x_{i_k - 1}) - f(t_{i_k})(G(x_{i_k}) - G(x_{i_k - 1}))\big| \\
&\leq \varepsilon \left( \sum_{k=1}^m (G(x_{i_k}) - G(x_{i_k - 1})) + \sum_{k=1}^m (G(x_{i_k - 1}) - L_G(x_{i_k - 1})) \right).
\end{align*}
Since $L_G(x_{i_k - 1}) \geq G(x_{i_{k-1} - 1})$ for each $k$, and noting that $x_{i_1 - 1} = 0_K$ implies $L_G(x_{i_1 - 1}) = 0$, we obtain:
\[
\sum_{k=1}^m \big|F(x_{i_k}) - F(x_{i_k - 1}) - f(t_{i_k})(G(x_{i_k}) - G(x_{i_k - 1}))\big|
\leq \varepsilon (2G(1_K) - G(0_K)).
\]

Next, let $t_{i_1} \leq t_{i_2} \leq \dots \leq t_{i_m}$ be all left-dense tag points of $P$ outside $E$. Then $G(t_{i_k}) = L_G(t_{i_k})$ for each $k$. Since $(x_{i_k - 1}, x_{i_k}] \subset \delta(t_{i_k})$ and $x_{i_k - 1} \leq t_{i_k} \leq x_{i_k}$, the Straddle Lemma~\ref{it:Straddle1} yields:
\[
|F(x_{i_k}) - F(x_{i_k - 1}) - f(t_{i_k})(G(x_{i_k}) - G(x_{i_k - 1}))| \leq \varepsilon (G(x_{i_k}) - G(x_{i_k - 1})).
\]
Thus,
\begin{align*}
\sum_{k=1}^m \big|F(x_{i_k}) &- F(x_{i_k - 1}) - f(t_{i_k})(G(x_{i_k}) - G(x_{i_k - 1}))\big|\leq \varepsilon (G(1_K) - G(0_K)).
\end{align*}

Combining all the estimates, we conclude:
\[
\sum_{i=1}^n \big|F(x_i) - F(x_{i-1}) - f(t_i)(G(x_i) - G(x_{i-1}))\big| \leq \varepsilon (1 + 3G(1_K) - 2G(0_K)).
\]

Therefore:
\begin{align*}
\big|F(1_K) - (F(0_K) - &f(0_K)G(0_K))- S(f, G, P)\big| \\
&= \left|F(1_K) - F(0_K) - \sum_{i=1}^n f(t_i)(G(x_i) - G(x_{i-1}))\right| \\
&= \left|\sum_{i=1}^n (F(x_i) - F(x_{i-1})) - \sum_{i=1}^n f(t_i)(G(x_i) - G(x_{i-1}))\right| \\
&\leq \sum_{i=1}^n \big|F(x_i) - F(x_{i-1}) - f(t_i)(G(x_i) - G(x_{i-1}))\big| \\
&\leq \varepsilon (1 + 3G(1_K) - 2G(0_K)).
\end{align*}
Since $\varepsilon > 0$ was arbitrary, the theorem is proved.
\end{proof}

\begin{example}
The following simple example shows that the assumption that points of non-$G$-differentiability are left-dense cannot, in general, be omitted in the previous theorem. Indeed, let $K$ be the compact line with two points $\{0,1\}$, and define functions $F$ and $G$ on $K$ by $F(0) = 2$, $F(1) = 3$, and $G(0) = G(1) = 1$. Then $F$ is $G$-differentiable at $0$, with $\frac{dF}{dG}(0) = 2$, but it is not $G$-differentiable at $1$.

Define $f: K \to \mathbb{R}$ by $f(0) = 2$ and $f(1) = 3$. Then
\[\int_K f \, dG = 2 \quad \text{and} \quad F(1) - \left(F(0) - f(0) G(0)\right) = 3 - (2 - 2 \cdot 1) = 3,\]
which shows that Theorem \ref{Thm:integration} fails in this case.
\end{example}

Building on \cite[Theorem 5.9]{bartle2001modern} (see also \cite[Theorem 6.5]{pouso2015new}), we present an alternative formulation of the Fundamental Theorem of Calculus in the context of compact lines.

\begin{theorem}[Fundamental Theorem of Calculus: differentiating  integrals]\label{Thm:differentiation}
Let $K$ be a compact line, and let $G: K \to \mathbb{R}$ be a non-decreasing, positive, and amenable function. Suppose that $f:K\to\mathbb{R}$ is $G$-integrable, and consider the function $F:K\to \mathbb{R}$ given by the formula
\[F(x) = \int_{0_K}^x f \, dG.\]
 Then there exists a $G$-null set $Z \subset K$ such that, for all $x \in K \setminus Z$, the derivative $\frac{dF}{dG}(x)$ exists and satisfies  
\[\frac{dF}{dG}(x) = f(x).\]
\end{theorem}
\begin{proof}
According to Theorem~\ref{Thm:Amenability}, the primitive function $F$ is amenable in $K$. Therefore, by Proposition~\ref{Pro:AuXDiff}, for every $x \in K$ such that $G(x) \neq L_G(x)$, the function $F$ is $G$-differentiable at $x$, and, once again by Theorem~\ref{Thm:Amenability}, we have
\[\frac{dF}{dG}(x) = \frac{F(x) - L_F(x)}{G(x) - L_G(x)} = f(x).\]

We deduce that the set $Z$ of all points $x \in K$ where either $\frac{dF}{dG}(x)$ does not exist, or it exists but differs from $f(x)$, must satisfy $G(x) = L_G(x)$. Let us define $U = \{x \in K : G(x) = L_G(x)\}$.
From \ref{it:diff_defin1} and \ref{it:diff_defin2}, we can partition $Z$ into two subsets $A$ and $B$ where: 
\begin{itemize}
    \item $A$ consists of all points $x \in U$ for which there exists an neighborhood $V$ of $x$ such that $G$ is constant on $V$, and
\item  $B$ is composed of all points $x \in U \setminus A$ for which there exists $\alpha(x) > 0$ such that, for every neighborhood $V$ of $x$, there exists $c_{x,V} \in V \setminus \{x\}$ satisfying
\begin{equation}\label{RelAuxNotDif}\left|F(x) - F(c_{x,V}) - f(x)\big(G(x) - G(c_{x,V})\big)\right| 
> \alpha(x) \left|G(x) - G(c_{x,V})\right|.
\end{equation}
\end{itemize}
To prove the theorem, it suffices to show that both sets $A$ and $B$ are $G$-null.

\medskip
 \textbf{Claim 1:} $A$ is $G$-null.
\medskip

For each $x \in A$, there exists a neighborhood $V$ of $x$ such that $F$ is constant on $V$. Thus, we may fix a filter basis at $x$, denoted by $\mathcal{F}_x$, consisting of closed intervals with the following properties:  
\begin{itemize}
    \item If $x$ is left-isolated, then $\mathcal{F}_x$ consists of intervals of the form $[x, z]$ with $x \leq z$ and $G(z) = G(x)$.
    \item If $x$ is left-dense, then $\mathcal{F}_x$ consists of intervals of the form $[a, b]$ with $a \leq x \leq b$ and $G(b) = L_G(a) = G(x)$.
\end{itemize}
Since each $\mathcal{F}_x$ is a filter basis of closed subsets of $K$, it follows that $\mathcal{F} = \bigcup_{x \in A} \mathcal{F}_x$ is an admissible covering for $A$ in the sense of Definition \ref{Def:AdmissibleCovering}. Thus, by Theorem \ref{Thm:vitali}, given $\varepsilon > 0$, there exists a finite, pairwise disjoint collection $\mathcal{H} = \{J_1, \dots, J_n\} \subset \mathcal{F}$ such that $\mu^*_G \left( A \setminus \bigcup \mathcal{H} \right) < \varepsilon.$

By the construction of the elements of $\mathcal{F}$, we have  
\[\mu(J_1 \cup \dots \cup J_n) = \mu(J_1) + \dots + \mu(J_n) = 0 + \dots + 0 = 0.\]  
Therefore, $\mu^*_G(A) < \varepsilon$, and since $\varepsilon > 0$ is arbitrary, Claim 1 is established.

\medskip
\textbf{Claim 2:} $B$ is $G$-null.
\medskip

Let $C$ be the subset consisting of all points $x \in B$ such that, for every open interval $V$ containing $x$, there exists $c_{x,V} > x$ satisfying relation~\eqref{RelAuxNotDif}. In this case, we fix $y_{x,V} = x$ and $z_{x,V} = c_{x,V}$, and define the interval $J_{x,V} = [y_{x,V}, z_{x,V}]$.

On the other hand, if $x \in B \setminus C$, then for every open interval $V$ containing $x$, there exists $c_{x,V} < x$ such that relation~\eqref{RelAuxNotDif} holds. In this case, we set $y_{x,V} = c_{x,V}$ and $z_{x,V} = x$. If $y_{x,V} \in C$, we define $J_{x,V} = [y_{x,V}, z_{x,V}]$; otherwise, we set $J_{x,V} = (y_{x,V}, z_{x,V}]$. 

An important observation about the intervals $J_{x,V}$ is that, if $\mu_G$ denotes the Radon measure induced by $G$, then for all $x$ and all open intervals $V$ containing $x$, we have
\begin{equation}\label{Rel:MeasureFinal}
\mu_G(J_{x,V}) = G(z_{x,V}) - G(y_{x,V}).
\end{equation}

Next, for each $n\in\mathbb N$, let $B_n=\{x\in B:\alpha(x)\geq 1/n\}$. Since $B=\bigcup_{n=1}^\infty B_n$, to conclude the proof it suffices to show that each $B_n$ is $G$-null. 
    
Fix $n \in \mathbb{N}$, and let $\varepsilon > 0$ be arbitrary. Since $f$ is $G$-integrable, we may fix a $\frac{\varepsilon}{n}$-gauge $\delta$ for $f$ (see Remark \ref{Rem:EpsilonGauge2}) and consider the collection  
\[\mathcal{F} = \{J_{x,V} : x \in B_n \text{ and } V\subset \delta(x)\text{ is an open interval containing } x\}.\]  
We claim that this family forms an admissible covering of $B_n$. Indeed, let $J_{x_1,V_1} < J_{x_2,V_2} < \dots < J_{x_m,V_m}$ be pairwise disjoint elements of $\mathcal{F}$, and let $a \in B_n \setminus \bigcup_{i=1}^m J_{x_i,V_i}$.
If there exists an open interval $V$ containing $a$ such that $V \cap \left( \bigcup_{i=1}^m J_{x_i,V_i} \right) = \emptyset$, then we are done, since for $V' = V \cap \delta(a)$, we have $J_{a,V'} \in \mathcal{F}$. Assume instead that every open interval $V$ containing $a$ intersects $\bigcup_{i=1}^m J_{x_i,V_i}$. By the construction of the sets $J_{x,V}$, this implies that there exists $1 \leq i \leq m$ such that $J_{x_i,V_i}$ is a semi-open interval of the form $(y_{x_i,V_i}, z_{x_i,V_i}]$ and $a = y_{x_i,V_i}$.
It follows that $0_K < a < 1_K$, that $a$ is left-dense, and that $a \notin C$. If $i = 1$, let $V = \delta(a)$; if $i > 1$, let $V = (z_{x_{i-1},V_{i-1}}, z_{x_i,V_{i}}) \cap \delta(a)$. By our construction, the interval $J_{a,V}$ belongs to $\mathcal{F}$ and is contained in the gap between the intervals $J_{x_i,V_i}$. More precisely, we have
\[J_{a,V} \cap \left( \bigcup_{i=1}^m J_{x_i,V_i} \right) = \emptyset,\]
which completes the proof that $\mathcal{F}$ is admissible.

Next, according to Theorem \ref{Thm:vitali}, there exists a finite subcollection $\mathcal{H} = \{J_{x_1,V_1}, \dots, J_{x_m,V_m}\} \subset \mathcal{F}$ consisiting of pairwise disjoint intervals such that $\mu_G^*(E_n \setminus \bigcup \mathcal{H}) < \varepsilon$,
where $\mu_G^*$ denotes the exterior measure induced by $\mu_G$.
On one hand, by employing Theorem~\ref{Thm:Additivity}, we obtain
\begin{align*}
    &\sum_{i=1}^m\left|f(x_i)(G(z_{x_i,V_i}) - G(y_{x_i,V_i})) + f(y_{x_i,V_i})G(y_{x_i,V_i}) - \int_{y_{x_i,V_i}}^{z_{x_i,V_i}} f\,dG\right|\\
    &= \sum_{i=1}^m\left|f(x_i)(G(z_{x_i,V_i}) - G(y_{x_i,V_i})) - (F(z_{x_i,V_i}) - F(y_{x_i,V_i}))\right| \\
    &> \sum_{i=1}^m \alpha(x_i)\big(G(z_{x_i,V_i}) - G(y_{x_i,V_i})\big) \\
    &\geq \frac{1}{n} \sum_{i=1}^m \big(G(z_{x_i,V_i}) - G(y_{x_i,V_i})\big).
\end{align*}

On the other hand, observe that $\mathcal{S} = \{([y_{x_1,V_1}, z_{x_1,V_1}], x_1), \dots, ([y_{x_m,V_m}, z_{x_m,V_m}], x_m)\}$ is a $\delta$-fine tagged system in $K$. It then follows from Lemma~\ref{lem:SH} (Saks--Henstock) that
\[\sum_{i=1}^m \left|f(x_i)(G(z_{x_i,V_i}) - G(y_{x_i,V_i})) + f(y_{x_i,V_i})G(y_{x_i,V_i}) - \int_{y_{x_i,V_i}}^{z_{x_i,V_i}} f\,dG\right| \leq \frac{4\varepsilon}{n}.\]

Combining these inequalities and recalling relation~\eqref{Rel:MeasureFinal}, we get
\[\mu_G\left(\bigcup \mathcal{H}\right) = \sum_{i=1}^m \big(G(z_{x_i,V_i}) - G(y_{x_i,V_i})\big) < 4\varepsilon.\]
Therefore,
\[\mu_G^*(E_n) \leq \mu_G^*\left(E_n \setminus \bigcup \mathcal{H}\right) + \mu_G^*\left(\bigcup \mathcal{H}\right) < 5\varepsilon,\]
which concludes the proof.
\end{proof}

\section{Convergence theorems and applications}\label{sec:convergence}

In what follows, we adapt the classical convergence theorems (see \cite[Section 8]{bartle2001modern}) to our integration setting. As an application, we obtain a characterization of Lebesgue integrability with respect to positive Radon measures in terms of the presented integral. Let us  begin with the following result:

\begin{theorem}[Monotone Convergence]\label{Thm:MCT}
Let $K$ be a compact line and $G: K \to \mathbb{R}$ a nondecreasing positive amenable function. Let $(f_m)_{m\in \mathbb{N}}$ be a monotone sequence of $G$-integrable functions $f_m: K \to \mathbb{R}$ that converges pointwise $G$-almost everywhere to a function $f: K \to \mathbb{R}$. If the sequence $\left(\int_K f_m \, dG\right)_{m}$ is convergent, then $f$ is also $G$-integrable on $K$ and the following equality holds:
\[\lim_{m\to \infty} \int_K f_m \, dG = \int_K f \, dG.\]
\end{theorem}
\begin{proof}
Let $V$ be the set of points where the sequence does not converge pointwise to $f$. Since $V$ is $G$-null, we may assume without loss of generality that $V \neq \emptyset$. By Proposition~\ref{prop:Gae}, we can redefine each function $f_m$, for $m \in \mathbb{N}$, as well as $f$, to be zero on $V$. After this modification, the convergence holds pointwise on all of $K$.

Let $\varepsilon > 0$ be arbitrary. Without loss of generality, assume that the sequence $(f_m)_{m \in \mathbb{N}}$ is nondecreasing. Since $\left(\int_K f_m \, dG\right)_{m \in \mathbb{N}}$ is convergent, we can define 
\[A = \lim_{m \to \infty} \int_K f_m \, dG = \sup_{m \in \mathbb{N}} \int_K f_m \, dG.\] 
Choose $N_0 \in \mathbb{N}$ such that 
\[\left| A - \int_K f_m \, dG \right| < \frac{\varepsilon}{3}\]
for all $m \geq N_0$.
Since each function $f_m$ is $G$-integrable, there exists a gauge $\delta_m$ such that whenever $P \ll \delta_m$, we have
\[\left| S(f_m,G,P) - \int_K f_m \, dG \right| < \frac{\varepsilon}{3\cdot 2^{m+1}}.\]
Because $(f_m)_{m}$ converges pointwise to $f$ on the compact space $K$, for each $x \in K$, there exists an integer $m(x)>N_0$ such that 
\[|f_{m(x)}(x) - f(x)| < \frac{\varepsilon}{3(G(1_K)+1)}.\]
Define a gauge $\delta$ by $\delta(x) = \delta_{m(x)}(x)$ for each $x \in K$. If $P =\{([x_{i-1},x_i],t_i): 1\leq i \leq n\}$ is a $\delta$-fine partition, we have the following estimation:
\begin{align*}
\left| S(f,G,P) - A \right| 
&\leq \left| S(f,G,P) -\left(f_{m(0_K)}(0_K) G(0_K)+\sum_{i=1}^n f_{m(t_i)}(t_i) (G(x_{i}) - G(x_{i-1})) \right)\right| \\
&\quad + \left| \sum_{i=1}^n \left(f_{m(t_i)}(t_i)(G(x_{i}) - G(x_{i-1}))+f_{m(t_i)}(x_{i-1}) G(x_{i-1}) - \int_{x_{i-1}}^{x_i} f_{m(t_i)} \, dG\right) \right| \\
&\quad + \left|f_{m(0_K)}(0_K) G(0_K)+\sum_{i=1}^n\left(\int_{x_{i-1}}^{x_i} f_{m(t_i)}\, dG -f_{m(t_{i})}(x_{i-1}) G(x_{i-1})\right)-A\right| .
\end{align*}

Let us label each term on the right-hand side of the relation above as $I_1$, $I_2$, and $I_3$, in that order. We will now analyze each of these terms individually. For the first term, we have
\begin{align*}
I_1&\leq |f(0_K) - f_{m(0_K)}(0_K)| G(0_K)+\sum_{i=1}^n |f(t_i) - f_{m(t_i)}(t_i)| (G(x_{i}) - G(x_{i-1}))\\
&<\frac{\varepsilon}{3(G(1_K)+1)}G(1_K)<\frac{\varepsilon}{3}.
\end{align*}
For the second term, let $p_1, \ldots, p_s$ be integers such that the collection $\{\varGamma_1, \ldots, \varGamma_s\}$, where  
$\varGamma_j = \{i \in \{1, \ldots, n\} : m(t_i) = p_j\}$ for each $1 \leq j \leq s$, forms a partition of $\{1, \ldots, n\}$.
For each $1 \leq j \leq s$, the collection $\mathcal{S}_j = \{([x_{i-1}, x_i], t_i) : i \in \varGamma_j\}$ forms a tagged system. Therefore, by applying Lemma~\ref{lem:SH} (Saks--Henstock), we obtain, for each $j \in {1, \ldots, s}$:
\[\left| \sum_{i \in \varGamma_j} \left( f_{p_j}(t_i)(G(x_{i}) - G(x_{i-1})) + f_{p_j}(x_{i-1}) G(x_{i-1}) - \int_{x_{i-1}}^{x_i} f_{p_j} \, dG \right) \right| < \frac{\varepsilon}{3 \cdot 2^{p_j}}.\]
It follows that 
\begin{align*}
I_2 &\leq \sum_{j=1}^k\frac{\varepsilon}{3 \cdot 2^{p_j}}<\frac{\varepsilon}{3}\sum_{j=1}^\infty\frac{\varepsilon}{2^{j}}=\frac{\varepsilon}{3}.
\end{align*}
  For the third term, we set $J_0=\{0_K\}$, and for each $1\leq i\leq n$, we define $J_i=(x_{i-1},x_i]$. Observe that, for each $1\leq i \leq n$, Propositions~\ref{prop:IntegrationOnASubinterval} and~\ref{Prop:int singleton} yield:  
\[f_{m(0_K)}(0_K) G(0_K) = \int_K f_{m(t_0)}|_{J_0} \, dG,\]
and for each $i \in \{1, \ldots, n\}$,
\begin{align*}
\int_{x_{i-1}}^{x_i} f_{m(t_i)} \, dG & - f_{m(t_i)}(x_{i-1}) G(x_{i-1})= \left( \int_{x_{i-1}}^{x_i} f_{m(t_i)} \, dG - f_{m(t_i)}(x_{i-1}) L_G(x_{i-1}) \right) \\
& \quad - f_{m(t_i)}(x_{i-1}) \left( G(x_{i-1}) - L_G(x_{i-1}) \right)= \int_K f_{m(t_i)}|_{J_i} \, dG.
\end{align*}
Therefore, we can write
\[
I_3 = \left| \sum_{i=0}^n \int_K f_{m(t_i)}|_{J_i} \, dG - A \right|.
\]

Now, let $N_1 = \max\{m(t_1), \ldots, m(t_n)\}$. Since $(f_m)_m$ is a nondecreasing sequence and $G$ is a nondecreasing positive function, we have:
\[\int_K f_{N_0}|_{J_i} \, dG \leq \int_K f_{m(t_i)}|_{J_i} \, dG \leq \int_K f_{N_1}|_{J_i} \, dG.\]
Using the additivity of the integral, we obtain:
\[
0 \leq A - \int_K f_{N_1} \, dG \leq A - \sum_{i=0}^n \int_K f_{m(t_i)}|_{J_i} \, dG \leq A - \int_K f_{N_0} \, dG < \frac{\varepsilon}{3}.
\]
Thus, we conclude that $I_3 < \frac{\varepsilon}{3}$.

Gathering all the information above, we have:
\[\left| S(f,G,P) - A \right|\leq I_1 + I_2 + I_3< \frac{\varepsilon}{3} + \frac{\varepsilon}{3} + \frac{\varepsilon}{3}= \varepsilon.\]

Hence, we deduce that $f$ is $G$-integrable and that $A = \int_K f \, dG$.
\end{proof}

With the Monotone Convergence Theorem at our disposal, we can readily adapt \cite[Theorem 8.7]{bartle2001modern} and \cite[Theorem 8.8]{bartle2001modern} to our integration framework.

\begin{lemma}[Fatou's Lemma]\label{Lem:FatouLemma}
Let $K$ be a compact line, and let $G: K \to \mathbb{R}$ be a non-decreasing, positive, and amenable function. Consider a sequence $(f_m)_{m \in \mathbb{N}}$ of $G$-integrable functions $f_m: K \to \mathbb{R}$ such that $g \leq f_m$ for every $m \in \mathbb{N}$, where $g$ is a $G$-integrable function. Then, $\liminf_{m} f_m$ is $G$-integrable, and 
\[\int_K\liminf_{m} f_m \, dG = \liminf_{m} \int_K f_m \, dG.\]
\end{lemma}
\begin{proof}
Reasoning as in \cite[Theorem 8.6]{bartle2001modern}, but employing Corollary~\ref{Cor:|f|IsIntegrable} and Theorem~\ref{Thm:MCT}, one can show that the function $\varphi_m: K \to \mathbb{R}$ defined by $\varphi_m(x) = \inf\{f_n(x) : n \geq m\}$ is $G$-integrable for each $m \in \mathbb{N}$, and satisfies
\[\int_K g \, dG \leq \int_K \varphi_m \, dG \leq \int_K f_m \, dG.\]
Since the sequence $\left(\int_K \varphi_m \, dG \right)_{m}$ is bounded and monotone, it converges to a real number. Moreover, as the sequence $(\varphi_m)_m$ is monotone and converges pointwise to $\liminf_m f_m$, the conclusion follows by another application of Theorem~\ref{Thm:MCT}.
\end{proof}

\begin{theorem}[Dominated Convergence]\label{Thm:Dominated}
Let $K$ be a compact line, and let $G: K \to \mathbb{R}$ be a non-decreasing, positive, and amenable function. Consider a sequence $(f_m)_{m \in \mathbb{N}}$ of $G$-integrable functions $f_m: K \to \mathbb{R}$ such that $g\leq f_m \leq h$ for every $m \in \mathbb{N}$, where $g,h$ are $G$-integrable functions. If $(f_m)_m$ converges pointwise to a function $f$, then 
$f$ is $G$-integrable and holds 
\[\lim_{m\to \infty} \int_K f_m \, dG = \int_K f \, dG.\]
\end{theorem}
\begin{proof}
We proceed as in \cite[Theorem 8.8]{bartle2001modern}, making use of Lemma~\ref{Lem:FatouLemma}.
\end{proof}

In this paper, given a compact line $K$, a function $f: K \to \mathbb{R}$ is said to be \emph{measurable} if it is measurable with respect to the Borel $\sigma$-algebra on $K$. A \emph{simple function} is a measurable function with finite image; that is, there exist scalars $c_1, \ldots, c_m \in \mathbb{R}$ such that $f[K] = \{c_1, \ldots, c_m\}$ and each level set $f^{-1}[\{c_k\}]$ belongs to the Borel $\sigma$-algebra on $K$ for $1 \leq k \leq m$. As a consequence of Theorem~\ref{Thm:Dominated}, we obtain the following standard lemma.

\begin{lemma}\label{Lem:BorelCharMeasurable}
Let $K$ be a compact line, and let $G: K \to \mathbb{R}$ be a non-decreasing, positive, and amenable function, and let $\mu_G\in \mathcal{M}(K)$ be the Radon measure induced by $G$. Then every simple function $f: K\to\mathbb{R}$ is $G$-integrable, and one has
\[\int_K f\,dG \,=\, \int_K f\,d\mu_G.\]
\end{lemma}
\begin{proof}
Since the integral is linear, it suffices to treat the case $f=\chi_B$ for a Borel set $B\subseteq K$.  By outer regularity of $\mu_G$, for each $n\in\mathbb{N}$ there exists an open set $U_n=\bigcup_{k=1}^\infty I_{k}^n$
which is a countable union of disjoint open intervals such that $B\subseteq U_n$, and satisfies
\[\mu_G\bigl(U_n\setminus B\bigr)<\tfrac1n.\]
Truncating to finitely many intervals, choose $N_n$ so that $J_n = \bigcup_{k=1}^{N_n} I_{k}^n$
satisfies $\mu_G\bigl(B\triangle J_n\bigr)<\tfrac1n$, and define the step function $\varphi_n = \chi_{J_n}$.  Then
\[\int_K|\chi_B-\varphi_n|\,d\mu_G = \mu_G(B\triangle J_n)<\tfrac1n.\]
By the Riesz--Fischer theorem (see \cite[Section~7.3]{royden2010real}), there exists a subsequence $(\varphi_{n_k})_k$ converging pointwise $G$-almost everywhere to $\chi_B$.  Since $|\varphi_{n_k}|\le1$, Theorem~\ref{Thm:Dominated} (Dominated Convergence) implies
\[\int_K\chi_B\,dG = \lim_{k\to\infty} \int_K\varphi_{n_k}\,dG = \lim_{k\to\infty} \int_K\varphi_{n_k}\,d\mu_G = \int_K\chi_B\,d\mu_G.\]
This completes the proof. 
\end{proof}

\begin{remark}\label{Rem:ModKurz}
Under the hypotheses of the previous lemma, it follows that for every measurable function $ f:K\to\mathbb{R} $, if $ |f| $ is $ G $-integrable, then $ f $ is also $ G $-integrable. Indeed, we can fix a sequence $ (\varphi_n)_n $ of simple functions converging pointwise to $ f $, with $ -|f| \leq \varphi_n \leq |f| $ for every $ n\in\mathbb{N} $. Then, using Lemma~\ref{Lem:BorelCharMeasurable} and then applying Theorem~\ref{Thm:Dominated}, we deduce that $ f $ is $ G $-integrable.
\end{remark}

We are now in a position to establish the following characterization, which relates our notion of integration to Lebesgue integration with respect to Radon measures.

\begin{theorem}\label{Thm:CharacterizationGRadon}
Let $K$ be a compact line, and let $G: K \to \mathbb{R}$ be a non-decreasing, positive, and amenable function, and let $\mu_G\in\mathcal{M}(K)$ denote the Radon measure induced by $G$.  Then a measurable function $f\colon K\to\mathbb{R}$ is Lebesgue $\mu_G$-integrable if and only if both $f$ and $|f|$ are $G$-integrable.  In that case,
\[\int_K f\,dG \;=\; \int_K f\,d\mu_G.\]
\end{theorem}
\begin{proof}
 Suppose $f$ is $\mu_G$-integrable.  Then $f^+=\max\{f,0\}$ is also $\mu_G$-integrable and nonnegative.  By standard measure-theoretic approximation there exists a nondecreasing sequence of simple functions converging pointwise to $f^+$ with $0\le\varphi_n\le f^+$ for all $n$.  By Lemma~\ref{Lem:BorelCharMeasurable}, each $\varphi_n$ is $G$-integrable, and the Monotone Convergence Theorem~\ref{Thm:MCT} gives
\[\int_K f^+\,dG
=\lim_{n\to\infty}\int_K\varphi_n\,dG
=\lim_{n\to\infty}\int_K\varphi_n\,d\mu_G
=\int_K f^+\,d\mu_G.\]
An identical argument applies to $f^-=-\min\{f,0\}$, so $f=f^+-f^-$ and $|f|=f^++f^-$ are both $G$-integrable.

Conversely, assume that $f$ and $|f|$ are $G$-integrable. By Remark~\ref{Rem:ModKurz}, both $f^+$ and $f^-$ are then $G$-integrable. Let $(\varphi_n)_n$ be a nondecreasing sequence of simple functions converging pointwise to $f^+$, with $0 \leq \varphi_n \leq f^+$ for every $n$. Arguing as in the previous case, we find that the sequence $\left(\int_{K} \varphi_n\, dG\right)_n$ converges to $\int_{K} f^+\, dG$, and therefore to a real number. Moreover, by Lemma~\ref{Lem:BorelCharMeasurable}, we have $\int_{K} \varphi_n\, d\mu_G = \int_{K} \varphi_n\, dG
\quad \text{for every } n \in \mathbb{N}$.
Applying the Monotone Convergence Theorem for Radon measures, we conclude that $f^+$ is $\mu_G$-integrable and that
\[\int_{K} f^+\, d\mu_G = \int_{K} f^+\, dG.\]
Repeating the same argument for $f^-$, we obtain the desired result by the additivity of the integral.
\end{proof}

\section*{Appendix}

 Here we detail the comparison of our integral to the classic Kurzweil--Stieltjes integral on a real interval  $[a,b]$, as it was presented for example in \cite[Definition 6.2.2]{monteiro2019kurzweil}, 
and its generalization (due to Peterson and Thompson, see \cite{peterson2006henstock}) to time scales known as the Kurzweil--Stieltjes nabla (or $\nabla$) integral. It will suffice to compare with the Kurzweil--Stieltjes $\nabla$-integral on time scales, since it generalizes the former. We recall the basic terminology behind the definition. A time scale is a nonempty set of the form $[a,b]\cap\mathbb T$, where $\mathbb T$ is a closed subset of $\mathbb R$. We can assume that $\min [a,b]\cap\mathbb T=a$ and $\max [a,b]\cap\mathbb T = b$, and let us recall that for each subinterval $I\subset [a,b]$ we denote $I_\mathbb T = I\cap \mathbb T$. The \emph{backward jump operator} on $[a,b]_\mathbb T$  is the function $\rho : [a,b]_\mathbb T  \to [a,b]_\mathbb T$ defined by
$$
    \rho(x)=  \begin{cases}
        \sup\{y\in [a,b]_\mathbb T:y<x\}, &\text{if }x\in (a,b]_\mathbb T\\
        a,& \text{if } x=a. 
    \end{cases}
$$
A \emph{$\nabla$-gauge} on $[a,b]_\mathbb T$ is a pair of real functions $\gamma=(\gamma_L,\gamma_R)$ on $[a,b]_\mathbb T$ satisfying
\begin{itemize}
    \item for each $x\in [a,b)_\mathbb T$, $\gamma_R(x)>0$,
    \item for each $x\in (a,b]_\mathbb T$, $\gamma_L(x)>0$ and $\gamma_L(x)\geq x-\rho(x)$, and
    \item $\gamma_L(a)\geq 0$ and $\gamma_R(b)\geq 0$. 
\end{itemize}
 Given $\nabla$-gauge $\gamma= (\gamma_L,\gamma_R)$, a partition $P = \{([x_{i-1},x_i],t_i):1\leq i \leq n\}$ is said to be  $\gamma$-fine if $[x_{i-1},x_i]\subset [t_i-\gamma_L(t_i),t_i+\gamma_R(t_i)]$. A function $f: [a,b]_\mathbb T\to \mathbb R$ is said to be  Kurzweil--Stieltjes $\nabla$-integrable on $[a,b]_\mathbb T$ with respect to some $G:[a,b]_\mathbb T\to\mathbb R$ if there is an $A\in\mathbb R$ such that, for each $\varepsilon >0$ there is a $\nabla$-gauge $\gamma$ on $[a,b]_\mathbb T$ such that, for every $\gamma$-fine partition $P$ of $[a,b]_\mathbb T$, $|S_\mathrm{RS}(f,G,P) - A|<\varepsilon$, where $S_\mathrm{RS}(f,G,P)$ is the ``classic'' Riemann--Stieltjes sum without the summand $f(a)G(a)$, that is,
 $$
 S_\mathrm{RS}(f,G,P)= \sum_{i=1}^n f(t_i)(G(x_i)-G(x_{i-1})).
 $$
 In affirmative case, we denote $A=\int_a^bf(x)\,d\nabla G(x)$. The relationship between this integral and the one presented in this work, when restricted to a time scale setting, is established as follows:

\begin{prop*}
    Let $f,G:[a,b]_\mathbb T\to\mathbb R$. If $f$ is  Kurzweil--Stieltjes $\nabla$-integrable with respect to $G$, then it is $G$-integrable in the sense of Definition \ref{def:integrability} and moreover 
    $$
    \int_a^bf\,dG = f(a)G(a)+ \int_a^bf(x)\,d\nabla G(x).
    $$
    In particular, when $G(a)=0$, the integral presented in Definition \ref{def:integrability} generalizes the Kurzweil--Stieltjes $\nabla$-integral on time scales. 
\end{prop*}    
\begin{proof}
    Suppose that $f$ is Kurzweil--Stieltjes $\nabla$-integrable with $\int_a^bf(x)\,d\nabla x=A$, fix $\varepsilon>0$, and let $\gamma$ be a $\nabla$-gauge such that $P\ll\gamma\Rightarrow |S_{\mathrm{RS}}(f,G,P)-A|<\varepsilon$. For each $x\in (a,b)_\mathbb T$, choose $r(x)>0$ such that
    $$
    x-\gamma_L(x)<x-r(x)<x<x+r(x)<x+\gamma_R(x).
    $$
    Fix also $r(a)$ and $r(b)$ such that $a<a+r(a)<a+\gamma_R(a)$ and $b-\gamma_L(b)<b-r(b)<b$. Consider the interval gauge on $[a,b]_\mathbb T$ given by $\delta(x)=(x-r(x),x+r(x))\cap [a,b]_\mathbb T$. Given an arbitrary $\delta$-fine partition $P=\{([x_{i-1},x_i],t_i):1\leq i \leq n\}$, note that $(x_{i-1},x_i]_{\mathbb{T}} \subset (x-r(x),x+r(x))_{\mathbb T}$ implies that $[x_{i-1},x_i] \subset [x-\gamma_L(x),x+\gamma_R(x)]_{\mathbb{T}}$ for each $i\in\{1,\dots,n\}$. Thus $P\ll \gamma$, from which
    $$
    |S(f,G,P)-(f(a)G(a) + A)| = |S_\mathrm{RS}(f,G,P)+f(a)G(a)-(f(a)G(a) + A)|<\varepsilon,
    $$
    which concludes the proof. 
\end{proof}

\bibliographystyle{siam}
\bibliography{references}

\end{document}